\documentclass[twoside,11pt]{article}

\usepackage{jmlr2e}

\usepackage{amsfonts, bbm, enumitem, parskip, mathtools}

\DeclareMathOperator*{\argmin}{arg\,min}
\DeclareMathOperator{\bor}{\mathcal{B}}
\DeclareMathOperator{\cov}{cov}
\DeclareMathOperator{\expect}{\mathbb{E}}

\DeclareMathOperator{\norm}{N}

\DeclareMathOperator{\prob}{\mathbb{P}}
\DeclareMathOperator{\rk}{rk}

\DeclareMathOperator{\spn}{sp}

\DeclareMathOperator{\tr}{tr}
\DeclareMathOperator{\trans}{\mathsf{T}}
\DeclareMathOperator{\vary}{var}

\newcommand{\comp}{\mathbb{C}}
\newcommand{\eps}{\varepsilon}
\newcommand{\iid}{i.i.d.\@ }

\newcommand{\real}{\mathbb{R}}
\newcommand{\subs}{\subseteq}

\ShortHeadings{Ivanov-Regularised Estimators over RKHSs}{Page and Gr\"{u}new\"{a}lder}
\firstpageno{1}

\begin{document}

\title{Ivanov-Regularised Least-Squares Estimators over Large RKHSs and Their Interpolation Spaces}

\author{\name Stephen Page \email s.page@lancaster.ac.uk \\
\addr STOR-i \\
Lancaster University\\
Lancaster, LA1 4YF, United Kingdom
\AND
\name Steffen Gr\"{u}new\"{a}lder \email s.grunewalder@lancaster.ac.uk \\
\addr Department of Mathematics and Statistics\\
Lancaster University\\
Lancaster, LA1 4YF, United Kingdom}

\editor{}

\maketitle

\begin{abstract}
We study kernel least-squares estimation under a norm constraint. This form of regularisation is known as Ivanov regularisation and it provides better control of the norm of the estimator than the well-established Tikhonov regularisation. Ivanov regularisation can be studied under minimal assumptions. In particular, we assume only that the RKHS is separable with a bounded and measurable kernel. We provide rates of convergence for the expected squared $L^2$ error of our estimator under the weak assumption that the variance of the response variables is bounded and the unknown regression function lies in an interpolation space between $L^2$ and the RKHS. We then obtain faster rates of convergence when the regression function is bounded by clipping the estimator. In fact, we attain the optimal rate of convergence. Furthermore, we provide a high-probability bound under the stronger assumption that the response variables have subgaussian errors and that the regression function lies in an interpolation space between $L^\infty$ and the RKHS. Finally, we derive adaptive results for the settings in which the regression function is bounded.
\end{abstract}

\begin{keywords}
Interpolation Space, Ivanov Regularisation, Regression, RKHS, Training and Validation
\end{keywords}

\section{Introduction} \label{sIntro}

One of the key problems to overcome in nonparametric regression is overfitting, due to estimators coming from large hypothesis classes. To avoid this phenomenon, it is common to ensure that both the empirical risk and some regularisation function are small when defining an estimator. There are three natural ways to achieve this goal. We can minimise the empirical risk subject to a constraint on the regularisation function, minimise the regularisation function subject to a constraint on the empirical risk or minimise a linear combination of the two. These techniques are known as Ivanov regularisation, Morozov regularisation and Tikhonov regularisation respectively \citep{oneto2016tikhonov}. Ivanov and Morozov regularisation can be viewed as dual problems, while Tikhonov regularisation can be viewed as the Lagrangian relaxation of either.

Tikhonov regularisation has gained popularity as it provides a closed-form estimator in many situations. In particular, Tikhonov regularisation in which the estimator is selected from a reproducing-kernel Hilbert space (RKHS) has been extensively studied \citep{smale2007learning, caponnetto2007optimal, steinwart2008support, mendelson2010regularization, steinwart2009optimal}. Although Tikhonov regularisation produces an estimator in closed form, it is Ivanov regularisation which provides the greatest control over the hypothesis class, and hence over the estimator it produces. For example, if the regularisation function is the norm of the RKHS, then the bound on this function forces the estimator to lie in a ball of predefined radius inside the RKHS. An RKHS norm measures the smoothness of a function, so the norm constraint bounds the smoothness of the estimator. By contrast, Tikhonov regularisation provides no direct control over the smoothness of the estimator.

The control we have over the Ivanov-regularised estimator is useful in many settings. The most obvious use of Ivanov regularisation is when the regression function lies in a ball of known radius inside the RKHS. In this case, Ivanov regularisation can be used to constrain the estimator to lie in the same ball. Suppose, for example, that we are interested in estimating the trajectory of a particle from noisy observations over time. Assume that the velocity or acceleration of the particle is constrained by certain physical conditions. Constraints of this nature can be imposed by bounding the norm of the trajectory in a Sobolev space. Certain Sobolev spaces are RKHSs, so it is possible to use Ivanov regularisation to enforce physical conditions on an estimator of the trajectory which match those of the trajectory itself. Ivanov regularisation can also be used within larger inference methods. It is compatible with validation, allowing us to control an estimator selected from an uncountable collection. This is because the Ivanov-regularised estimator is continuous in the size of the ball containing it (see Lemma \ref{lContEst}), so the estimators parametrised by an interval of ball sizes can be controlled simultaneously using chaining.

In addition to the other useful properties of the Ivanov-regularised estimator, Ivanov regularisation can be performed almost as quickly as Tikhonov regularisation. The Ivanov-regularised estimator is a support vector machine (SVM) with regularisation parameter selected to match the norm constraint (see Lemma \ref{lCalcEst}). This parameter can be selected to within a tolerance $\eps$ using interval bisection with order $\log(1/\eps)$ iterations. In general, Ivanov regularisation requires the calculation of order $\log(1/\eps)$ SVMs.

In this paper, we study the behaviour of the Ivanov-regularised least-squares estimator with regularisation function equal to the norm of the RKHS. We derive a number of novel results concerning the rate of convergence of the estimator in various settings and under various assumptions. Our analysis is performed by controlling empirical processes over balls in the RKHS. By contrast, the analysis of Tikhonov-regularised estimators usually relies on the spectral decomposition of the kernel operator $T$ on $L^2 (P)$. Here, $P$ is the covariate distribution.

We first prove an expectation bound on the squared $L^2 (P)$ error of our estimator of order $n^{- \beta/2}$, under the weak assumption that the response variables have bounded variance. Here, $n$ is the number of data points, and $\beta$ parametrises the interpolation space between $L^2 (P)$ and $H$ containing the regression function. As far as we are aware, the analysis of an estimator in this setting has not previously been considered. The definition of an interpolation space is given in Section \ref{sRKHSInter}. The expected squared $L^2 (P)$ error can be viewed as the expected squared error of our estimator at a new independent covariate, with the same distribution $P$. If we also assume that the regression function is bounded, then it makes sense to clip our estimator so that it takes values in the same interval as the regression function. This further assumption allows us to achieve an expectation bound on the squared $L^2 (P)$ error of the clipped estimator of order $n^{- \beta/(1 + \beta)}$.

We then move away from the average behaviour of the error towards its behaviour in the worst case. We obtain high-probability bounds of the same order, under the stronger assumption that the response variables have subgaussian errors and the interpolation space is between $L^\infty$ and $H$. The second assumption is quite natural as we already assume that the regression function is bounded, and $H$ can be continuously embedded in $L^\infty$ since it has a bounded kernel $k$. Note that this assumption means that the set of possible regression functions is independent of the covariate distribution.

When the regression function is bounded, we also analyse an adaptive version of our estimator, which does not require us to know which interpolation space contains the regression function. This adaptive estimator obtains bounds of the same order as the non-adaptive one. The adaptive estimator is created using training and validation. The bounds do not contain higher-order terms because we treat the validation step as a form of Ivanov regularisation, as opposed to, for example, using a union bound.

Our expectation bound of order $n^{- \beta/(1 + \beta)}$, when the regression function is bounded, improves on the high-probability bound of \citet{smale2007learning} of order $n^{-\beta/2}$. Their bound is attained under the stronger assumption that the regression function lies in the image of a power of the kernel operator, instead of an interpolation space \citep[see][]{steinwart2012mercer}. The authors also assume that the response variables are bounded. Furthermore, for a fixed $\beta \in (0, 1)$, \citet{steinwart2009optimal} show that there is an instance of our problem with a bounded regression function such that the following holds. For all estimators $\hat f$ of $g$, for some $\eps > 0$, we have
\begin{equation*}
\lVert \hat f - g \rVert_{L^2 (P)}^2 \geq C_{\alpha, \eps} n^{- \alpha}
\end{equation*}
with probability at least $\eps$ for all $n \geq 1$, for some constant $C_{\alpha, \eps} > 0$, for all $\alpha > \beta/(1 + \beta)$. Hence, for all estimators $\hat f$ of $g$, we have
\begin{equation*}
\expect \left (\lVert \hat f - g \rVert_{L^2 (P)}^2 \right ) \geq C_{\alpha, \eps} \eps n^{- \alpha}
\end{equation*}
for all $n \geq 1$, for all $\alpha > \beta/(1 + \beta)$. In this sense, our expectation bound in this setting is optimal because it attains the order $n^{- \beta/(1 + \beta)}$, the smallest possible power of $n$. Our expectation bound on the adaptive version of our estimator is also optimal, because the bound is of the same order as in the easier non-adaptive setting.

The high-probability bound of \citet{steinwart2009optimal} is optimal in a similar sense, although the authors achieve faster rates by assuming a fixed rate of decay of the eigenvalues of the kernel operator $T$, as discussed in Section \ref{sLit}. Since there is an additional parameter for the decay of the eigenvalues, the collection of problem instances for a fixed set of parameters is smaller in their paper. This means that our optimal rates are the slowest of the optimal rates in \citet{steinwart2009optimal}.

Note that all of the bounds in this paper are on the squared $L^2 (P)$ error of our estimators. We do not consider errors based on other norms, such as the interpolation space norms, as it is unclear how the techniques used in this paper could be adapted to such errors. It seems likely that a completely different approach would be necessary to bound these other errors.

\section{RKHSs and Their Interpolation Spaces} \label{sRKHSInter}

A Hilbert space $H$ of real-valued functions on $S$ is an RKHS if the evaluation functional $L_x : H \to \real$, $L_x h = h(x)$, is bounded for all $x \in S$. In this case, $L_x \in H^*$ the dual of $H$ and the Riesz representation theorem tells us that there is some $k_x \in H$ such that $h(x) = \langle h, k_x \rangle_H$ for all $h \in H$. The kernel is then given by $k(x_1, x_2) = \langle k_{x_1}, k_{x_2} \rangle_H$ for $x_1, x_2 \in S$, and is symmetric and positive-definite.

Now suppose that $(S, \mathcal{S})$ is a measurable space on which $P$ is a probability measure. We can define a range of interpolation spaces between $L^2 (P)$ and $H$ \citep{bergh2012interpolation}. Let $(Z, \lVert \cdot \rVert_Z)$ be a Banach space and $(V, \lVert \cdot \rVert_V)$ be a dense subspace of $Z$. The $K$-functional of $(Z, V)$ is
\begin{equation*}
K(z, t) = \inf_{v \in V} (\lVert z - v \rVert_Z + t \rVert v \lVert_V)
\end{equation*}
for $z \in Z$ and $t > 0$. For $\beta \in (0, 1)$ and $1 \leq q < \infty$, we define
\begin{equation*}
\lVert z \rVert_{\beta, q} = \left ( \int_0^\infty ( t^{-\beta} K(z, t) )^q t^{-1} dt \right )^{1/q} \text{ and } \lVert z \rVert_{\beta, \infty} = \sup_{t > 0} (t^{-\beta} K(z, t))
\end{equation*}
for $z \in Z$. The interpolation space $[Z, V]_{\beta, q}$ is defined to be the set of $z \in Z$ such that $\lVert z \rVert_{\beta, q} < \infty$. Smaller values of $\beta$ give larger spaces. The space $[Z, V]_{\beta, q}$ is not much larger than $V$ when $\beta$ is close to 1, but we obtain spaces which get closer to $Z$ as $\beta$ decreases. The following result is essentially Theorem 3.1 of \citet{smale2003estimating}. The authors only consider the case in which $\lVert v \rVert_Z \leq \lVert v \rVert_V$ for all $v \in V$, however the result holds by the same proof even without this condition.

\begin{lemma} \label{lApproxInter}
Let $(Z, \lVert \cdot \rVert_Z)$ be a Banach space, $(V, \lVert \cdot \rVert_V)$ be a dense subspace of $Z$ and $z \in [Z, V]_{\beta, \infty}$. We have
\begin{equation*}
\inf \{ \lVert v - z \rVert_Z : v \in V, \lVert v \rVert_V \leq r \} \leq \frac{\lVert z \rVert_{\beta, \infty}^{1/(1-\beta)}}{r^{\beta/(1-\beta)}}.
\end{equation*}
\end{lemma}

When $H$ is dense in $L^2 (P)$, we can define the interpolation spaces $[L^2 (P), H]_{\beta, q}$, where $L^2 (P)$ is the space of measurable functions $f$ on $(S, \mathcal{S})$ such that $f^2$ is integrable with respect to $P$. We work with $q = \infty$, which gives the largest space of functions for a fixed $\beta \in (0, 1)$. We can then use the approximation result in Lemma \ref{lApproxInter}. When $H$ is dense in $L^\infty$, we also require $[L^\infty, H]_{\beta, q}$, where $L^\infty$ is the space of bounded measurable functions on $(S, \mathcal{S})$.

\section{Literature Review} \label{sLit}

Early research on RKHS regression does not make assumptions on the rate of decay of the eigenvalues of the kernel operator. For example, \citet{smale2007learning} assume that the response variables are bounded and the regression function is of the form $g = T^{\beta/2} f$ for $\beta \in (0, 1]$ and $f \in L^2 (P)$. Here, $T : L^2 (P) \to L^2 (P)$ is the kernel operator and $P$ is the covariate distribution. The authors achieve a squared $L^2 (P)$ error of order $n^{- \beta/2}$ with high probability by using SVMs. 

Initial research which does make assumptions on the rate of decay of the eigenvalues of the kernel operator, such as that of \citet{caponnetto2007optimal}, assumes that the regression function is at least as smooth as an element of $H$. However, their paper still allows for regression functions of varying smoothness by letting $g \in T^{(\beta - 1)/2} (H)$ for $\beta \in [1, 2]$. By assuming that the $i$th eigenvalue of $T$ is of order $i^{- 1/p}$ for $p \in (0, 1]$, the authors achieve a squared $L^2 (P)$ error of order $n^{- \beta / (\beta + p)}$ with high probability by using SVMs. This squared $L^2 (P)$ error is shown to be of optimal order for $\beta \in (1, 2]$.

Later research focuses on the case in which the regression function is at most as smooth as an element of $H$. Often, this research demands that the response variables are bounded. For example, \citet{mendelson2010regularization} assume that $g \in T^{\beta/2} (L^2 (P))$ for $\beta \in (0, 1)$ to obtain a squared $L^2 (P)$ error of order $n^{- \beta / (1 + p)}$ with high probability by using Tikhonov-regularised least-squares estimators. The authors also show that if the eigenfunctions of the kernel operator $T$ are uniformly bounded in $L^\infty$, then the order can be improved to $n^{- \beta / (\beta + p)}$. \citet{steinwart2009optimal} relax the condition on the eigenfunctions to the condition
\begin{equation*}
\lVert h \rVert_\infty \leq C_p \lVert h \rVert_H^p \lVert h \rVert_{L^2 (P)}^{1-p}
\end{equation*}
for all $h \in H$ and some constant $C_p > 0$. The same rate is attained by using clipped Tikhonov-regularised least-squares estimators, including clipped SVMs, and is shown to be optimal. The authors assume that $g$ is in an interpolation space between $L^2 (P)$ and $H$, which is slightly more general than the assumption of \citet{mendelson2010regularization}. A detailed discussion about the image of $L^2 (P)$ under powers of $T$ and interpolation spaces between $L^2 (P)$ and $H$ is given by \citet{steinwart2012mercer}.

Lately, the assumption that the response variables must be bounded has been relaxed to allow for subexponential errors. However, the assumption that the regression function is bounded has been maintained. For example, \citet{fischer2017sobolev} assume that $g \in T^{\beta/2} (L^2 (P))$ for $\beta \in (0, 2]$ and that $g$ is bounded. The authors also assume that $T^{\alpha/2} (L^2 (P))$ is continuously embedded in $L^\infty$, with respect to an appropriate norm on $T^{\alpha/2} (L^2 (P))$, for some $\alpha < \beta$. This gives the same squared $L^2 (P)$ error of order $n^{- \beta / (\beta + p)}$ with high probability by using SVMs.

\section{Contribution}

In this paper, we provide bounds on the squared $L^2 (P)$ error of our Ivanov-regularised least-squares estimator when the regression function comes from an interpolation space between $L^2 (P)$ and an RKHS $H$, which is separable with a bounded and measurable kernel $k$. We use the norm of the RKHS as our regularisation function. Under the weak assumption that the response variables have bounded variance, we prove a bound on the expected squared $L^2 (P)$ error of order $n^{- \beta/2}$ (Theorem \ref{tInterBound} on page \pageref{tInterBound}). As far as we are aware, the analysis of an estimator in this setting has not previously been considered. If we assume that the regression function is bounded, then we can clip the estimator and achieve an expected squared $L^2 (P)$ error of order $n^{- \beta/(1+\beta)}$ (Theorem \ref{tBddInterBound} on page \pageref{tBddInterBound}).

Under the stronger assumption that the response variables have subgaussian errors and the regression function comes from an interpolation space between $L^\infty$ and $H$, we show that the squared $L^2 (P)$ error is of order $n^{- \beta/(1+\beta)}$ with high probability (Theorem \ref{tProbBddInterBound} on page \pageref{tProbBddInterBound}). For the settings in which the regression function is bounded, we use training and validation on the data in order to select the size of the constraint on the norm of our estimator. This gives us an adaptive estimation result which does not require us to know which interpolation space contains the regression function. We obtain a squared $L^2 (P)$ error of order $n^{- \beta/(1+\beta)}$ in expectation and with high probability, depending on the setting (Theorems \ref{tValInterBound} and \ref{tProbValInterBound} on pages \pageref{tValInterBound} and \pageref{tProbValInterBound}). In order to perform training and validation, the response variables in the validation set must have subgaussian errors. The expectation results for bounded regression functions are of optimal order in the sense discussed at the end of Section \ref{sIntro}.
\begin{table}[h]
\begin{center}
\begin{tabular}{l|ccc}
Regression Function & $L^2 (P)$ Interpolation & $L^\infty$ Bound & $L^\infty$ Interpolation \\
Response Variables & Bounded Variance & Bounded Variance & Subgaussian Errors \\
Bound Type & Expectation & Expectation & High Probability \\
\hline
Bound Order & $n^{-\beta/2}$ & $n^{-\beta/(1 + \beta)}$ $(\ast)$ & $n^{-\beta/(1 + \beta)}$ \\
\end{tabular}
\end{center}
\caption{Orders of bounds on squared $L^2 (P)$ error}
\label{tRes}
\end{table}

The validation results are summarised in Table \ref{tValRes}. Again, the columns for which there is an $L^\infty$ bound on the regression function also make the $L^2 (P)$ interpolation assumption. The assumptions on the response variables relate to those in the validation set, which has $\tilde n$ data points. We assume that $\tilde n$ is equal to some multiple of $n$. Again, orders of bounds marked with $(\ast)$ are known to be optimal.

\begin{table}[h]
\begin{center}
\begin{tabular}{l|cc}
Regression Function & $L^\infty$ Bound & $L^\infty$ Interpolation \\
Response Variables & Subgaussian Errors & Subgaussian Errors \\
Bound Type & Expectation & High Probability \\
\hline
Bound Order & $n^{-\beta/(1 + \beta)}$ $(\ast)$ & $n^{-\beta/(1 + \beta)}$
\end{tabular}
\end{center}
\caption{Orders of validation bounds on squared $L^2 (P)$ error}
\label{tValRes}
\end{table}

The validation results are summarised in Table \ref{tValRes}. Again, the columns for which there is an $L^\infty$ bound on the regression function also make the $L^2 (P)$ interpolation assumption. The assumptions on the response variables relate to those in the validation set, which has $\tilde n$ data points. We assume that $\tilde n$ is equal to some multiple of $n$. Again, orders of bounds marked with $(\ast)$ are known to be optimal.

\section{Problem Definition} \label{sDef}

We now formally define our regression problem. For a topological space $T$, let $\bor(T)$ be the Borel $\sigma$-algebra of $T$. Let $(S, \mathcal{S})$ be a measurable space. Assume that $(X_i, Y_i)$ for $1 \leq i \leq n$ are $(S \times \real, \mathcal{S} \otimes \bor(\real))$-valued random variables on the probability space $(\Omega, \mathcal{F}, \prob)$, which are \iid with $X_i \sim P$ and $\expect(Y_i^2) < \infty$, where $\expect$ denotes integration with respect to $\prob$. Since any version of $\expect(Y_i \vert X_i)$ is $\sigma(X_i)$-measurable, where $\sigma(X_i)$ is the $\sigma$-algebra generated by $X_i$, we have that $\expect(Y_i \vert X_i) = g(X_i)$ almost surely for some function $g$ which is measurable on $(S, \mathcal{S})$ \citep[Section A3.2 of][]{williams1991probability}. From the definition of conditional expectation and the identical distribution of the $(X_i, Y_i)$, it is clear that we can choose $g$ to be the same for all $1 \leq i \leq n$. The conditional expectation used is that of Kolmogorov, defined using the Radon--Nikodym derivative. Its definition is unique almost surely. Since $\expect(Y_i^2) < \infty$, it follows that $g \in L^2 (P)$ by Jensen's inequality. To summarise, $\expect(Y_i \vert X_i) = g(X_i)$ almost surely for $1 \leq i \leq n$ with $g \in L^2 (P)$. We assume throughout that
\begin{enumerate}[label={($Y$1)}]
\item \hfil $\vary(Y_i \vert X_i) \leq \sigma^2$ almost surely for $1 \leq i \leq n$. \label{Y1}
\end{enumerate}

\vspace{0.5\baselineskip}

Our results depend on how well $g$ can be approximated by elements of an RKHS $H$ with kernel $k$. We make the following assumptions.
\begin{enumerate}[label={($H$)}]
\item {The RKHS $H$ with kernel $k$ has the following properties:
\begin{itemize}
\item The RKHS $H$ is separable.
\item The kernel $k$ is bounded.
\item The kernel $k$ is a measurable function on $(S \times S, \mathcal{S} \otimes \mathcal{S})$.
\end{itemize}} \label{H}
\end{enumerate}
We define
\begin{equation*}
\lVert k \rVert_\infty = \sup_{x \in S} k(x, x)^{1/2} < \infty.
\end{equation*}
We can guarantee that $H$ is separable by, for example, assuming that $k$ is continuous and $S$ is a separable topological space \citep[Lemma 4.33 of][]{steinwart2008support}. The fact that $H$ has a kernel $k$ which is measurable on $(S \times S, \mathcal{S} \otimes \mathcal{S})$ guarantees that all functions in $H$ are measurable on $(S, \mathcal{S})$ \citep[Lemma 4.24 of][]{steinwart2008support}.

\section{Ivanov Regularisation} \label{sIvanov}

We now consider Ivanov regularisation for least-squares estimators. Let $P_n$ be the empirical distribution of the $X_i$ for $1 \leq i \leq n$. The definition of Ivanov regularisation provides us with the following result.

\begin{lemma} \label{lIneq}
Let $A \subs L^2 (P)$. It may be that $A$ is a function of $\omega \in \Omega$ and does not contain $g$. Let
\begin{equation*}
\hat f \in \argmin_{f \in A} \frac{1}{n} \sum_{i = 1}^n (f(X_i) - Y_i)^2.
\end{equation*}
Then, for all $f \in A$ and all $\omega \in \Omega$, we have
\begin{equation*}
\lVert \hat f -  f \rVert_{L^2 (P_n)}^2 \leq \frac{4}{n} \sum_{i = 1}^n (Y_i - g(X_i)) (\hat f(X_i) - f(X_i)) + 4 \lVert f - g \rVert_{L^2 (P_n)}^2.
\end{equation*}
\end{lemma}

The function $g$ in Lemma \ref{lIneq} need not be the regression function, however this is the case of interest. In general, the first term of the right-hand side of the inequality must be controlled by bounding it with
\begin{equation}
\sup_{f_1, f_2 \in A} \frac{4}{n} \sum_{i = 1}^n (Y_i - g(X_i)) (f_1 (X_i) - f_2 (X_i)). \label{eSup1}
\end{equation}
This is usually not measurable. However, if $A$ is a fixed subset of a separable RKHS, then $A$ is separable and the function which evaluates $f \in A$ at $X_i$ is continuous for $1 \leq i \leq n$. This means that the supremum can be replaced with a countable supremum, so the quantity is a random variable on $(\Omega, \mathcal{F})$. Clearly, this term increases as $A$ gets larger. However, if $A$ gets larger, then we may select $f \in A$ closer to $g$. Hence, we can make the second term of the right-hand side of the inequality in Lemma \ref{lIneq} smaller. This demonstrates the trade-off in selecting the size of $A$ for the Ivanov-regularised least-squares estimator constrained to lie in $A$.

The next step in analysing $\hat f$ is to move to a bound on
\begin{equation}
\lVert \hat f -  f \rVert_{L^2 (P)}^2 \leq \lVert \hat f -  f \rVert_{L^2 (P_n)}^2 + \sup_{f_1, f_2 \in A} \left \lvert \lVert f_1 - f_2 \rVert_{L^2 (P_n)}^2 - \lVert f_1 - f_2 \rVert_{L^2 (P)}^2 \right \rvert. \label{eSup2}
\end{equation}
The second term on the right-hand side of this inequality is measurable when $A$ is a fixed subset of a separable RKHS. It also increases with $A$. Finally, we obtain a bound on
\begin{equation*}
\lVert \hat f -  g \rVert_{L^2 (P)}^2 \leq 2 \lVert \hat f -  f \rVert_{L^2 (P)}^2 + 2 \lVert f -  g \rVert_{L^2 (P)}^2.
\end{equation*}
This again demonstrates why $f \in A$ should be close to $g$.

\section{Estimator Definition}

In this section, we consider the Ivanov-regularised estimators of Section \ref{sIvanov} in the context of RKHS regression. Recall the RKHS $H$ from the definition of our regression problem in Section \ref{sDef}. Let $B_H$ be the closed unit ball of $H$ and $r > 0$. The Ivanov-regularised least-squares estimator constrained to lie in $r B_H$ is
\begin{equation*}
\hat h_r = \argmin_{f \in r B_H} \frac{1}{n} \sum_{i = 1}^n (f(X_i) - Y_i)^2.
\end{equation*}
We also define $\hat h_0 = 0$. The next result shows that $\hat h_r$ is an SVM with regularisation parameter $\mu (r)$.

\begin{lemma} \label{lCalcEst}
Assume \ref{H}. Let $K$ be the $n \times n$ symmetric matrix with $K_{i, j} = k(X_i, X_j)$. Then $K$ is an $(\real^{n \times n}, \bor(\real^{n \times n}))$-valued measurable matrix on $(\Omega, \mathcal{F})$ and there exist an orthogonal matrix $A$ and a diagonal matrix $D$ which are both $(\real^{n \times n}, \bor(\real^{n \times n}))$-valued measurable matrices on $(\Omega, \mathcal{F})$ such that $K = A D A^{\trans}$. Furthermore, the diagonal entries of $D$ are non-negative and non-increasing. Let $m = \rk K$, which is a random variable on $(\Omega, \mathcal{F})$. For $r > 0$, if
\begin{equation*}
r^2 < \sum_{i = 1}^m D_{i, i}^{-1} (A^{\trans} Y)_i^2,
\end{equation*}
then define $\mu(r) > 0$ by
\begin{equation}
\sum_{i = 1}^m \frac{D_{i, i}}{(D_{i, i} + n \mu(r))^2} (A^{\trans} Y)_i^2 = r^2. \label{eMuDef}
\end{equation}
Otherwise, let $\mu(r) = 0$. We have that $\mu(r)$ is strictly decreasing when $\mu(r) > 0$, and $\mu(r)$ is measurable on $(\Omega \times (0, \infty), \mathcal{F} \otimes \bor((0, \infty)))$, where $r$ varies in $(0, \infty)$. Let $a \in \real^n$ be defined by
\begin{equation*}
(A^{\trans} a)_i = (D_{i, i} + n \mu(r))^{-1} (A^{\trans} Y)_i
\end{equation*}
for $1 \leq i \leq m$ and $(A^{\trans} a)_i = 0$ for $m+1 \leq i \leq n$, noting that $A^{\trans}$ has the inverse $A$ since it is an orthogonal matrix. For $r \geq 0$, we can uniquely define $\hat h_r$ by demanding that $\hat h_r \in \spn \{ k_{X_i} : 1 \leq i \leq n \}$. This gives
\begin{equation*}
\hat h_r = \sum_{i = 1}^n a_i k_{X_i}
\end{equation*}
for $r > 0$ and $\hat h_0 = 0$. We have that $\hat h_r$ is a $(H, \bor(H))$-valued measurable function on $(\Omega \times [0, \infty), \mathcal{F} \otimes \bor([0, \infty)))$, where $r$ varies in $[0, \infty)$.
\end{lemma}

Let $r > 0$. There are multiple methods for calculating $\mu(r)$ to within a given tolerance $\eps > 0$. We call this value $\nu(r)$.

\subsection{Diagonalising $K$} \label{ssDiagK}

 Firstly, $\mu(r) = 0$ if and only if
\begin{equation*}
r \geq \left ( \sum_{i = 1}^m D_{i, i}^{-1} (A^{\trans} Y)_i^2 \right )^{1/2},
\end{equation*}
so in this case we set $\nu(r) = 0$. Otherwise, $\mu(r) > 0$ and
\begin{align*}
r^2 &= \sum_{i = 1}^m \frac{D_{i, i}}{(D_{i, i} + n \mu(r))^2} (A^{\trans} Y)_i^2 \\
&\leq n^{-2} \left ( \sum_{i = 1}^m D_{i, i} (A^{\trans} Y)_i^2 \right ) \mu(r)^{-2}.
\end{align*}
Hence,
\begin{equation}
\mu(r) \leq n^{-1} \left ( \sum_{i = 1}^m D_{i, i} (A^{\trans} Y)_i^2 \right )^{1/2} r^{-1}. \label{eMuBd}
\end{equation}
The function
\begin{equation*}
\sum_{i = 1}^m \frac{D_{i, i}}{(D_{i, i} + n \mu)^2} (A^{\trans} Y)_i^2
\end{equation*}
of $\mu \geq 0$ is continuous. Hence, we can calculate $\nu(r)$ using interval bisection on the interval with lower end point $0$ and upper end point equal to the right-hand side of \eqref{eMuBd}. We can then approximate $a$ by replacing $\mu(r)$ with $\nu(r)$ in the calculation of $a$ in Lemma \ref{lCalcEst}.

\subsection{Not Diagonalising $K$}

We can calculate an alternative $\nu(r)$ without diagonalising $K$. Note that if $\mu(r) > 0$, then \eqref{eMuDef} can be written as
\begin{equation*}
Y^{\trans} ( K + n \mu(r) I)^{-1} K ( K + n \mu(r) I)^{-1} Y = r^2.
\end{equation*}
Since $\mu(r)$ is strictly decreasing for $\mu(r) > 0$, we have
\begin{equation*}
r \geq \left (Y^{\trans} ( K + n \eps I)^{-1} K ( K + n \eps I)^{-1} Y \right )^{1/2}
\end{equation*}
if and only if $\mu(r) \in [0, \eps]$, so in this case we set $\nu(r) = \eps$. Otherwise, $\mu(r) > \eps$ and \eqref{eMuBd} can be written as
\begin{equation}
\mu(r) \leq n^{-1} (Y^{\trans} K Y)^{1/2} r^{-1}. \label{eMuBd2}
\end{equation}
The function
\begin{equation*}
Y^{\trans} ( K + n \mu I)^{-1} K ( K + n \mu I)^{-1} Y
\end{equation*}
of $\mu > 0$ is continuous. Hence, we can calculate $\nu(r)$ using interval bisection on the interval with lower end point $\eps$ and upper end point equal to the right-hand side of \eqref{eMuBd2}. When $\mu(r) > 0$ or $K$ is invertible, we can also calculate $a$ in Lemma \ref{lCalcEst} using $a = (K + n \mu(r) I)^{-1} Y$. Since $\nu(r) > 0$, we can approximate $a$ by $( K + n \nu(r) I)^{-1} Y$.

If we have that $K$ is invertible, then we can calculate the $\nu(r)$ in Subsection \ref{ssDiagK} while still not diagonalising $K$. We have $\mu(r) = 0$ if and only if $r \geq (Y^{\trans} K^{-1} Y)^{1/2}$, so in this case we set $\nu(r) = 0$. Otherwise, $\mu(r) > 0$ and \eqref{eMuBd} can be written as
\begin{equation*}
\mu(r) \leq n^{-1} (Y^{\trans} K Y)^{1/2} r^{-1},
\end{equation*}
so we can again use interval bisection to calculate $\nu(r)$.  We can still approximate $a$ by $(K + n \nu(r) I)^{-1} Y$.

\subsection{Approximating $\hat h_r$}

Having discussed how to approximate $\mu(r)$ by $\nu(r)$ to within a given tolerance $\eps > 0$, we now consider the estimator produced by this approximation. We find that this estimator is equal to $\hat h_s$ for some $s > 0$. We only have $\nu(r) = 0$ for the methods considered above when $\mu(r) = 0$, in which case we can let $s = r$ to obtain the approximate estimator $\hat h_s = \hat h_r$. Otherwise, $\nu(r) > 0$. Let 
\begin{equation*}
s = \left ( \sum_{i = 1}^m \frac{D_{i, i}}{(D_{i, i} + n \nu(r))^2} (A^{\trans} Y)_i^2 \right )^{1/2}.
\end{equation*}
By \eqref{eMuDef}, we have $\mu(s) = \nu(r)$ and the approximate estimator is equal to $\hat h_s$. Assume that $r$ is bounded away from 0 as $n \to \infty$ and let $C > 0$ be some constant not depending on $n$. We can ensure that $s$ is of the same order as $r$ as $n \to \infty$ by demanding that $s$ is within $C$ of $r$. This is enough to ensure that the orders of convergence for $\hat h_r$ apply to $\hat h_s$. In order to attain this value of $\nu(r)$, interval bisection should terminate at $x \in \real$ such that
\begin{equation*}
\left ( \sum_{i = 1}^m \frac{D_{i, i}}{(D_{i, i} + n x)^2} (A^{\trans} Y)_i^2 \right )^{1/2}
\end{equation*}
is within $C$ of $r$. Note that this guarantees $\lVert \hat h_s - \hat h_r \rVert_H \leq C^{1/2} (r + s)^{1/2}$ by Lemma \ref{lContEst}.

\section{Expectation Bounds}

To capture how well $g$ can be approximated by elements of $H$, we define
\begin{equation*}
I_2 (g, r) = \inf \left \{ \lVert h_r - g \rVert_{L^2 (P)}^2 : h_r \in r B_H \right \}
\end{equation*}
for $r > 0$. We consider the distance of $g$ from $r B_H$ because we constrain our estimator $\hat h_r$ to lie in this set. The supremum in \eqref{eSup1} with $A = r B_H$ can be controlled using the reproducing kernel property and the Cauchy--Schwarz inequality to obtain
\begin{equation*}
8 r \left ( \frac{1}{n^2} \sum_{i, j = 1}^n (Y_i - g(X_i)) (Y_j - g(X_j)) k(X_i, X_j) \right )^{1/2}.
\end{equation*}
The expectation of this quantity can be bounded using Jensen's inequality. Something very similar to this argument gives the first term of the bound in Theorem \ref{tBound} below. The expectation of the supremum in \eqref{eSup2} with $A = r B_H$ can be controlled using symmetrisation \citep[Lemma 2.3.1 of][]{van1996weak} to obtain
\begin{equation*}
2 \expect \left ( \sup_{f \in 2 r B_H} \left \lvert \frac{1}{n} \sum_{i = 1}^n \eps_i f(X_i)^2 \right \rvert \right ),
\end{equation*}
where the $\eps_i$ for $1 \leq i \leq n$ are \iid Rademacher random variables on $(\Omega, \mathcal{F}, \prob)$, independent of the $X_i$. Since $\lVert f \rVert_\infty \leq 2 \lVert k \lVert_\infty r$ for all $f \in 2 r B_H$, we can remove the squares on the $f(X_i)$ by using the contraction principle for Rademacher processes \citep[Theorem 3.2.1 of][]{gine2015mathematical}. This quantity can then be bounded in a similar way to the supremum in \eqref{eSup1}, giving the second term of the bound in Theorem \ref{tBound} below.

\begin{theorem} \label{tBound}
Assume \ref{Y1} and \ref{H}. Let $r > 0$. We have
\begin{equation*}
\expect \left (\lVert \hat h_r -  g \rVert_{L^2 (P)}^2 \right ) \leq \frac{8 \lVert k \rVert_\infty \sigma r}{n^{1/2}} + \frac{64 \lVert k \rVert_\infty^2 r^2}{n^{1/2}} + 10 I_2 (g, r).
\end{equation*}
\end{theorem}

We can obtain rates of convergence for our estimator $\hat h_r$ if we make an assumption about how well $g$ can be approximated by elements of $H$. Let us assume
\begin{enumerate}[label={($g$1)}]
\item \hfil $g \in [L^2 (P), H]_{\beta, \infty}$ with norm at most $B$ for $\beta \in (0, 1)$ and $B > 0$. \label{g1}
\end{enumerate}
The assumption \ref{g1}, together with Lemma \ref{lApproxInter}, give
\begin{equation}
I_2 (g, r) \leq \frac{B^{2/(1-\beta)}}{r^{ 2 \beta/(1-\beta)}} \label{eApproxInter}
\end{equation}
for $r > 0$. We obtain an expectation bound on the squared $L^2 (P)$ error of our estimator $\hat h_r$ of order $n^{-\beta/2}$.

\begin{theorem} \label{tInterBound}
Assume \ref{Y1}, \ref{H} and \ref{g1}. Let $r > 0$. We have
\begin{equation*}
\expect \left (\lVert \hat h_r -  g \rVert_{L^2 (P)}^2 \right ) \leq \frac{8 \lVert k \rVert_\infty \sigma r}{n^{1/2}} + \frac{64 \lVert k \rVert_\infty^2 r^2}{n^{1/2}} + \frac{10 B^{2/(1-\beta)}}{r^{ 2 \beta/(1-\beta)}}.
\end{equation*}
Let $D_1 > 0$. Setting
\begin{equation*}
r = D_1  \lVert k \rVert_\infty^{- (1-\beta)} B n^{(1-\beta)/4}
\end{equation*}
gives
\begin{equation*}
\expect \left (\lVert \hat h_r -  g \rVert_{L^2 (P)}^2 \right ) \leq D_2 \lVert k \rVert_\infty^{2 \beta} B^2 n^{- \beta/2} +D_3 \lVert k \rVert_\infty^\beta B \sigma n^{-(1+\beta)/4}
\end{equation*}
for constants $D_2, D_3 > 0$ depending only on $D_1$ and $\beta$.
\end{theorem}

Since we must let $r \to \infty$ for the initial bound in Theorem \ref{tInterBound} to tend to $0$, the second term of the initial bound is asymptotically larger than the first. If we ignore the first term and minimise the second and third terms over $r > 0$, we get
\begin{equation*}
r = \left ( \frac{5 \beta}{32 (1-\beta)} \right )^{(1-\beta)/2} \lVert k \rVert_\infty^{- (1-\beta)} B n^{(1-\beta)/4}.
\end{equation*}
In particular, $r$ is of the form in Theorem \ref{tInterBound}. This choice of $r$ gives
\begin{equation*}
D_2 = 64 \left ( \frac{5 \beta}{32 (1-\beta)} \right )^{1-\beta} + 10 \left ( \frac{32 (1-\beta)}{5 \beta} \right )^\beta \text{ and } D_3 = 8 \left ( \frac{5 \beta}{32 (1-\beta)} \right )^{(1-\beta)/2}.
\end{equation*}

The fact that the second term of the initial bound is larger than the first produces some interesting observations. Firstly, the choice of $r$ above does not depend on $\sigma^2$. Secondly, we can decrease the bound if we can find a way to reduce the second term, without having to alter the other terms. The increased size of the second term is due to the fact that the bound on $f \in 2 r B_H$ is given by $\lVert f \rVert_\infty \leq 2 \lVert k \lVert_\infty r$ when applying the contraction principle for Rademacher processes. If we can use a bound which does not depend on $r$, then we can reduce the size of the second term.

We now also assume
\begin{enumerate}[label={($g$2)}]
\item \hfil $\lVert g \rVert_\infty \leq C$ for $C > 0$ \label{g2}
\end{enumerate}
and clip our estimator. Let $r > 0$. Since $g$ is bounded in $[- C, C]$, we can make $\hat h_r$ closer to $g$ by constraining it to lie in the same interval. Similarly to Chapter 7 of \citet{steinwart2008support} and \citet{steinwart2009optimal}, we define the projection $V : \real \to [-C, C]$ by
\begin{equation*}
V(t) =\left \{
\begin{array}{lll}
- C & \text{if} &t < - C \\
t & \text{if} &\lvert t \rvert \leq C \\
C & \text{if} &t > C
\end{array}
\right .
\end{equation*}
for $t \in \real$. We can apply the inequality
\begin{equation*}
\lVert V \hat h_r - V h_r \rVert_{L^2 (P_n)}^2 \leq \lVert \hat h_r - h_r \Vert_{L^2 (P_n)}^2
\end{equation*}
for all $h_r \in r B_H$. We continue analysing $V \hat h_r$ by bounding
\begin{equation*}
\sup_{f_1, f_2 \in r B_H} \left \lvert \lVert V f_1 - V f_2 \rVert_{L^2 (P_n)}^2 - \lVert V f_1 - V f_2 \rVert_{L^2 (P)}^2 \right \rvert.
\end{equation*}
The expectation of the supremum can be bounded in the same way as before, with some adjustments. After symmetrisation, we can remove the squares on the $V f_1 (X_i) - V f_2 (X_i)$ for $f_1, f_2 \in r B_H$ and $1 \leq i \leq n$ by using the contraction principle for Rademacher processes with $\lVert V f_1 - V f_2 \rVert_\infty \leq 2 C$. We can then use the triangle inequality to remove $V f_2 (X_i)$, before applying the contraction principle again to remove $V$. The expectation bound on the squared $L^2 (P)$ error of our estimator $V \hat h_r$ follows in the same way as before.

\begin{theorem} \label{tBddBound}
Assume \ref{Y1}, \ref{H} and \ref{g2}. Let $r > 0$. We have
\begin{equation*}
\expect \left (\lVert V \hat h_r -  g \rVert_{L^2 (P)}^2 \right ) \leq \frac{8 \lVert k \rVert_\infty (16 C + \sigma) r}{n^{1/2}} + 10 I_2 (g, r).
\end{equation*}
\end{theorem}

We can obtain rates of convergence for our estimator $V \hat h_r$ by again assuming \ref{g1}. We obtain an expectation bound on the squared $L^2 (P)$ error of $V \hat h_r$ of order $n^{-\beta/(1+\beta)}$.

\begin{theorem} \label{tBddInterBound}
Assume \ref{Y1}, \ref{H}, \ref{g1} and \ref{g2}. Let $r > 0$. We have
\begin{equation*}
\expect \left (\lVert V \hat h_r -  g \rVert_{L^2 (P)}^2 \right ) \leq \frac{8 \lVert k \rVert_\infty (16 C + \sigma) r}{n^{1/2}} + \frac{10 B^{2/(1-\beta)}}{r^{ 2 \beta/(1-\beta)}}.
\end{equation*}
Let $D_1 > 0$. Setting
\begin{equation*}
r = D_1 \lVert k \rVert_\infty^{- (1-\beta)/(1+\beta)} B^{2/(1+\beta)} (16 C + \sigma)^{- (1-\beta)/(1+\beta)} n^{(1-\beta)/(2 (1+\beta))}
\end{equation*}
gives
\begin{equation*}
\expect \left (\lVert V \hat h_r -  g \rVert_{L^2 (P)}^2 \right ) \leq D_2 \lVert k \rVert_\infty^{2 \beta/(1+\beta)} B^{2/(1+\beta)} (16 C + \sigma)^{2 \beta/(1+\beta)} n^{- \beta/(1+\beta)}
\end{equation*}
for a constant $D_2 > 0$ depending only on $D_1$ and $\beta$.
\end{theorem}

If we minimise the initial bound in Theorem \ref{tBddInterBound} over $r > 0$, we get
\begin{equation*}
r = \left ( \frac{5 \beta}{2 (1-\beta)} \right )^{(1-\beta)/(1+\beta)} \lVert k \rVert_\infty^{- (1-\beta)/(1+\beta)} B^{2/(1+\beta)} (16 C + \sigma)^{- (1-\beta)/(1+\beta)} n^{(1-\beta)/(2 (1+\beta))}.
\end{equation*}
In particular, $r$ is of the form in Theorem \ref{tBddInterBound}. This choice of $r$ gives
\begin{equation*}
D_2 = 2 \cdot 5^{(1-\beta)/(1+\beta)} \cdot 4^{2 \beta/(1+\beta)} \left ( \left ( \frac{2 \beta}{1 - \beta} \right )^{(1-\beta)/(1+\beta)} + \left ( \frac{1-\beta}{2 \beta} \right )^{2 \beta/(1+\beta)} \right ).
\end{equation*}

Although the second bound in Theorem \ref{tBddInterBound} is of theoretical interest, it is in practice impossible to select $r$ of the correct order in $n$ for the bound to hold without knowing $\beta$. Since assuming that we know $\beta$ is not realistic, we must use some other method for determining a good choice of $r$.

\subsection{Validation}

Suppose that we have an independent second data set $(\tilde X_i, \tilde Y_i)$ for $1 \leq i \leq \tilde n$ which are $(S \times \real, \mathcal{S} \otimes \bor(\real))$-valued random variables on the probability space $(\Omega, \mathcal{F}, \prob)$. Let the $(\tilde X_i, \tilde Y_i)$ be \iid with $\tilde X_i \sim P$ and $\expect(\tilde Y_i \vert \tilde X_i) = g(\tilde X_i)$ almost surely. Let $\rho \geq 0$ and $R \subs [0, \rho]$ be non-empty and compact. Furthermore, let $F = \{ V \hat h_r : r \in R \}$. We estimate a value of $r$ which makes the squared $L^2 (P)$ error of $V \hat h_r$ small by
\begin{equation*}
\hat r = \argmin_{r \in R} \frac{1}{\tilde n} \sum_{i = 1}^{\tilde n} (V \hat h_r (\tilde X_i) - \tilde Y_i)^2.
\end{equation*}
The minimum is attained because Lemma \ref{lContEst} shows that it is the minimum of a continuous function over a compact set. In the event of ties, we may take $\hat r$ to be the infimum of all points attaining the minimum. Lemma \ref{lValMeasEst} shows that the estimator $\hat r$ is a random variable on $(\Omega, \mathcal{F})$. Hence, by Lemma \ref{lCalcEst}, $\hat h_{\hat r}$ is a $(H, \bor(H))$-valued random variable on $(\Omega, \mathcal{F})$.

The definition of $\hat r$ means that we can analyse $V \hat h_{\hat r}$ using Lemma \ref{lIneq}. The expectation of the supremum in \eqref{eSup1} with $A = F$ can be bounded using chaining \citep[Theorem 2.3.7 of][]{gine2015mathematical}. The diameter of $(F, \lVert \cdot \rVert_\infty)$ is $2 C$, which is an important bound for the use of chaining. Hence, this form of analysis can only be performed under the assumption \ref{g2}. After symmetrisation, the expectation of the supremum in \eqref{eSup2} with $A = F$ can be bounded in the same way. In order to perform chaining, we need to make an assumption on the behaviour of the errors of the response variables $\tilde Y_i$ for $1 \leq i \leq \tilde n$. Let $U$ and $V$ be random variables on $(\Omega, \mathcal{F}, \prob)$. We say $U$ is $\sigma^2$-subgaussian if
\begin{equation*}
\expect(\exp(t U)) \leq \exp(\sigma^2 t^2 / 2)
\end{equation*}
for all $t \in \real$. We say $U$ is $\sigma^2$-subgaussian given $V$ if
\begin{equation*}
\expect(\exp(t U) \vert V) \leq \exp(\sigma^2 t^2 / 2)
\end{equation*}
almost surely for all $t \in \real$. We assume
\begin{enumerate}[label={($\tilde Y$)}]
\item \hfil $\tilde Y_i - g(\tilde X_i)$ is $\tilde \sigma^2$-subgaussian given $\tilde X_i$ for $1 \leq i \leq \tilde n$. \label{tY}
\end{enumerate}
This is stronger than the equivalent of the assumption \ref{Y1}, that $\vary(\tilde Y_i \vert \tilde X_i) \leq \tilde \sigma^2$ almost surely.

\begin{theorem} \label{tValBound}
Assume \ref{H} and \ref{tY}. Let $r_0 \in R$. We have
\begin{equation*}
\expect \left (\lVert V \hat h_{\hat r} -  g \rVert_{L^2 (P)}^2 \right )
\end{equation*}
is at most
\begin{equation*}
\frac{32 C (4 C + \tilde \sigma)}{\tilde n^{1/2}} \left ( \left ( 2 \log \left (2 + \frac{\lVert k \rVert_\infty^2 \rho^2}{C^2} \right ) \right )^{1/2} + \pi^{1/2} \right ) + 10 \expect \left ( \lVert V \hat h_{r_0} - g \rVert_{L^2 (P)}^2 \right ).
\end{equation*}
\end{theorem}

In order for us to apply the validation result in Theorem \ref{tValBound} to the initial bound in Theorem \ref{tBddInterBound}, we need to make an assumption on $R$. We assume either
\begin{enumerate}[label={($R$1)}]
\item \hfil $R = [0, \rho]$ for $\rho = a n^{1/2}$ and $a > 0$ \label{R1}
\end{enumerate}
or
\begin{enumerate}[label={($R$2)}]
\item \hfil $R = \{ b i : 0 \leq i \leq I - 1 \} \cup \{ a n^{1/2} \}$ and $\rho = a n^{1/2}$ for $a, b > 0$ and $I = \lceil a n^{1/2} / b \rceil$. \label{R2}
\end{enumerate}
The assumption \ref{R1} is mainly of theoretical interest and would make it difficult to calculate $\hat r$ in practice. The estimator $\hat r$ can be computed under the assumption \ref{R2}, since in this case $R$ is finite. We obtain an expectation bound on the squared $L^2 (P)$ error of $V \hat h_{\hat r}$ of order $n^{-\beta/(1+\beta)}$. This is the same order in $n$ as the second bound in Theorem \ref{tBddInterBound}.

\begin{theorem} \label{tValInterBound}
Assume \ref{Y1}, \ref{H}, \ref{g1}, \ref{g2} and \ref{tY}. Also assume \ref{R1} or \ref{R2} and that $\tilde n$ increases at least linearly in $n$. We have
\begin{equation*}
\expect \left (\lVert V \hat h_{\hat r} -  g \rVert_{L^2 (P)}^2 \right ) \leq D_1 n^{- \beta/(1+\beta)}
\end{equation*}
for a constant $D_1 > 0$ not depending on $n$ or $\tilde n$.
\end{theorem}

\section{High-Probability Bounds}

In this section, we look at how to extend our expectation bounds on our estimators to high-probability bounds. In order to do this, we must control the second term of the bound in Lemma \ref{lIneq} with $A = r B_H$ for $r > 0$, which is
\begin{equation}
\lVert h_r - g \rVert_{L^2 (P_n)}^2 \label{eIneqVary}
\end{equation}
for $h_r \in r B_H$. There is no way to bound \eqref{eIneqVary} in high-probability without strict assumptions on $g$. In fact, the most natural assumption is \ref{g2} that $\lVert g \rVert_\infty \leq C$ for $C > 0$, which we assume throughout this section. Bounding \eqref{eIneqVary} also requires us to introduce a new measure of how well $g$ can be approximated by elements of $H$. We define
\begin{equation*}
I_\infty (g, r) = \inf \left \{ \lVert h_r - g \rVert_\infty^2 : h_r \in r B_H \right \}
\end{equation*}
for $r > 0$. Note that $I_\infty (g, r) \geq I_2 (g, r)$. Using $I_\infty (g, r)$ instead of $I_2 (g, r)$ means that we do not have to control \eqref{eIneqVary} by relying on $\lVert h_r - g \rVert_\infty \leq \lVert k \rVert_\infty r + C$. Using Hoeffding's inequality, this would add a term of order $r^2 t^{1/2} / n^{1/2}$ for $t \geq 1$ to the bound in Theorem \ref{tProbBddBound} below, which holds with probability $1 - 3 e^{-t}$, substantially increasing its size.

It may be possible to avoid this problem by instead considering the Ivanov-regularised least-squares estimator
\begin{equation*}
\hat f_r = \argmin_{f \in  V(r B_H)} \frac{1}{n} \sum_{i = 1}^n (f(X_i) - Y_i)^2
\end{equation*}
for $r > 0$, where $V (r B_H) = \{ V h_r : h_r \in r B_H \}$. The second term of the bound in Lemma \ref{lIneq} with $A = V(r B_H)$ is
\begin{equation}
\lVert V h_r - g \rVert_{L^2 (P_n)}^2 \label{eIneqVary2}
\end{equation}
for $h_r \in r B_H$. Since $\lVert V h_r - g \rVert_\infty \leq 2 C$, using Hoeffding's inequality to bound \eqref{eIneqVary2} would only add a term of order $C^2 t^{1/2} / n^{1/2}$ to the bound in Theorem \ref{tProbBddBound} below, which would not alter its size. However, the calculation and analysis of the estimator $\hat f_r$ is outside the scope of this paper. This is because the calculation of $\hat f_r$ involves minimising a quadratic form subject to a series of linear constraints, and its analysis requires a bound on the supremum in \eqref{eSup1} with $A = V(r B_H)$.

The rest of the analysis of $V \hat h_r$ is similar to that of the expectation bound. The supremum in \eqref{eSup1} with $A = r B_H$ can again be bounded by
\begin{equation*}
8 r \left ( \frac{1}{n^2} \sum_{i, j = 1}^n (Y_i - g(X_i)) (Y_j - g(X_j)) k(X_i, X_j) \right )^{1/2}.
\end{equation*}
The quadratic form can be bounded using Lemma \ref{lSubgaussProb}, under an assumption on the behaviour of the errors of the response variables $Y_i$ for $1 \leq i \leq n$. The proof of Theorem \ref{tProbBddBound} below uses a very similar argument to this one. The supremum in \eqref{eSup2} with $A = r B_H$ can be bounded using Talagrand's inequality \citep[Theorem A.9.1 of][]{steinwart2008support}. In order to use Lemma \ref{lSubgaussProb}, we must assume
\begin{enumerate}[label={($Y$2)}]
\item \hfil $Y_i - g(X_i)$ is $\sigma^2$-subgaussian given $X_i$ for $1 \leq i \leq n$. \label{Y2}
\end{enumerate}
This assumption is stronger than \ref{Y1}. In particular, Theorem \ref{tBddBound} still holds under the assumptions \ref{Y2}, \ref{H} and \ref{g2}.

\begin{theorem} \label{tProbBddBound}
Assume \ref{Y2}, \ref{H} and \ref{g2}. Let $r > 0$ and $t \geq 1$. With probability at least $1 - 3 e^{-t}$, we have
\begin{equation*}
\lVert V \hat h_r -  g \rVert_{L^2 (P)}^2 \leq \left ( D_1 + D_2 r^{1/2} + D_3 r \right ) t^{1/2} n^{- 1/2} + D_4 t n^{-1} + 10 I_\infty (g, r)
\end{equation*}
for some constants $D_1, D_2, D_3, D_4 > 0$ not depending on $n$ or $t$.
\end{theorem}

We can obtain rates of convergence for our estimator $V \hat h_r$, but we must make a new assumption about how well $g$ can be approximated by elements of $H$, instead of \ref{g1}. We now assume
\begin{enumerate}[label={($g$3)}]
\item \hfil $g \in [L^\infty, H]_{\beta, \infty}$ with norm at most $B$ for $\beta \in (0, 1)$ and $B > 0$, \label{g3}
\end{enumerate}
instead of $g \in [L^2 (P), H]_{\beta, \infty}$ with norm at most $B$. This assumption is stronger than \ref{g1}, as it implies that the norm of $g \in [L^2 (P), H]_{\beta, \infty}$ is
\begin{equation*}
\sup_{t > 0} (t^{-\beta} \inf_{h \in H} (\lVert g - h \rVert_{L^2 (P)} + t \rVert h \lVert_H)) \leq \sup_{t > 0} (t^{-\beta} \inf_{h \in H} (\lVert g - h \rVert_{L^\infty} + t \rVert h \lVert_H)) \leq B.
\end{equation*}
In particular, Theorem \ref{tBddInterBound} still holds under the assumptions \ref{Y1}, \ref{H}, \ref{g2} and \ref{g3} or \ref{Y2}, \ref{H}, \ref{g2} and \ref{g3}. The assumption \ref{g3}, together with Lemma \ref{lApproxInter}, give
\begin{equation}
I_\infty (g, r) \leq \frac{B^{2/(1-\beta)}}{r^{ 2 \beta/(1-\beta)}}. \label{eLinftyApproxInter}
\end{equation}
We obtain a high-probability bound on the squared $L^2 (P)$ error of $V \hat h_r$ of order \linebreak $t^{\beta/(1+\beta)} n^{-\beta/(1+\beta)}$ with probability at least $1 - e^{- t}$.

\begin{theorem} \label{tProbBddInterBound}
Assume \ref{Y2}, \ref{H}, \ref{g2} and \ref{g3}. Let $r > 0$ and $t \geq 1$. With probability at least $1 - 3 e^{-t}$, we have
\begin{equation*}
\lVert V \hat h_r -  g \rVert_{L^2 (P)}^2 \leq \left ( D_1 + D_2 r^{1/2} + D_3 r \right ) t^{1/2} n^{- 1/2} + D_4 t n^{-1} + D_5 r^{- 2 \beta/(1-\beta)}
\end{equation*}
for some constants $D_1, D_2, D_3, D_4, D_5 > 0$ not depending on $n$ or $t$. Let $D_6 > 0$. Setting
\begin{equation*}
r = D_6 t^{- (1-\beta)/(2 (1+\beta))} n^{(1-\beta)/(2 (1+\beta))}
\end{equation*}
gives
\begin{equation*}
\lVert V \hat h_r -  g \rVert_{L^2 (P)}^2
\end{equation*}
is at most
\begin{equation*}
D_7 t^{\beta/(1+\beta)} n^{- \beta/(1+\beta)} + D_8 t^{(1+3\beta)/(4(1+\beta))} n^{- (1+3\beta)/(4(1+\beta))} + D_1 t^{1/2} n^{-1/2} + D_4 t n^{-1}
\end{equation*}
for constants $D_7, D_8 > 0$ not depending on $n$ or $t$.
\end{theorem}

Since we must let $r \to \infty$ for the initial bound in Theorem \ref{tProbBddInterBound} to tend to $0$, the asymptotically largest terms in the bound are
\begin{equation*}
D_3 r t^{1/2} n^{- 1/2} + D_5 r^{- 2 \beta/(1-\beta)}.
\end{equation*}
If we minimise this over $r > 0$, we get $r$ of the form in Theorem \ref{tProbBddInterBound} with
\begin{equation*}
D_6 = \left ( \frac{2 D_5 \beta}{D_3 (1-\beta)} \right )^{(1-\beta)/(1+\beta)}.
\end{equation*}
It is possible to avoid the dependence of $r$ on $t$ by setting
\begin{equation*}
r = D_6 n^{(1-\beta)/(2 (1+\beta))}
\end{equation*}
to obtain the bound
\begin{equation*}
D_7 t^{1/2} n^{- \beta/(1+\beta)} + D_8 t^{1/2} n^{- (1+3\beta)/(4(1+\beta))} + D_1 t^{1/2} n^{-1/2} + D_4 t n^{-1}.
\end{equation*}

\subsection{Validation}

We now extend our expectation bound on $V \hat h_{\hat r}$ to a high-probability bound. The supremum in \eqref{eSup1} with $A = F$ can be bounded using chaining \citep[Exercise 1 of Section 2.3 of][]{gine2015mathematical}, while the supremum in \eqref{eSup2} with $A = F$ can be bounded using Talagrand's inequality.

\begin{theorem} \label{tProbValBound}
Assume \ref{H} and \ref{tY}. Let $r_0 \in R$ and $t \geq 1$. With probability at least $1 - 3 e^{-t}$, we have
\begin{equation*}
\lVert V \hat h_{\hat r} -  g \rVert_{L^2 (P)}^2
\end{equation*}
is at most
\begin{align*}
&\frac{20 C (C + \tilde \sigma) t^{1/2}}{\tilde n^{1/2}} \left ( 1 + 32 \left ( \left ( 2 \log \left (2 + \frac{\lVert k \rVert_\infty^2 \rho^2}{C^2} \right ) \right )^{1/2} + \pi^{1/2} \right ) \right ) \\
+ \; &\frac{48 C^2 t^{1/2}}{\tilde n^{1/2}} + \frac{16 C^2 t}{3 \tilde n} + 10 \lVert V \hat h_{r_0} - g \rVert_{L^2 (P)}^2.
\end{align*}
\end{theorem}

We can apply the validation result in Theorem \ref{tProbValBound} to the initial bound in Theorem \ref{tProbBddInterBound} by assuming either \ref{R1} or \ref{R2}. We obtain a high-probability bound on the squared $L^2 (P)$ error of $V \hat h_{\hat r}$ of order $t^{1/2} n^{-\beta/(1+\beta)}$ with probability at least $1 - e^{- t}$. This is the same order in $n$ as the second bound in Theorem \ref{tProbBddInterBound}.

\begin{theorem} \label{tProbValInterBound}
Assume \ref{Y2}, \ref{H}, \ref{g2}, \ref{g3} and \ref{tY}. Let $t \geq 1$. Also assume \ref{R1} or \ref{R2} and that $\tilde n$ increases at least linearly in $n$. With probability at least $1 - 6 e^{-t}$, we have
\begin{equation*}
\lVert V \hat h_{\hat r} -  g \rVert_{L^2 (P)}^2 \leq D_1 t^{1/2} n^{-\beta/(1+\beta)} + D_2 t n^{-1}
\end{equation*}
for constants $D_1, D_2 > 0$ not depending on $n$, $\tilde n$ or $t$.
\end{theorem}

\section{Discussion}

In this paper, we show how Ivanov regularisation can be used to produce smooth estimators which have a small squared $L^2 (P)$ error. We first consider the case in which the regression function lies in an interpolation space between $L^2 (P)$ and an RKHS $H$. We achieve bounds on the squared $L^2 (P)$ error under the assumption that $H$ is separable with a bounded and measurable kernel. Under the weak assumption that the response variables have bounded variance, we prove an expectation bound on the squared $L^2 (P)$ error of our estimator of order $n^{- \beta/2}$. Here, $\beta$ parametrises the interpolation space between $L^2 (P)$ and $H$ containing the regression function. As far as we are aware, the analysis of an estimator in this setting has not previously been considered.

If we assume that the regression function is bounded, then we can clip the estimator and show that the clipped estimator has an expected squared $L^2 (P)$ error of order $n^{- \beta/(1+\beta)}$. Under the stronger assumption that the response variables have subgaussian errors and that the regression function comes from an interpolation space between $L^\infty$ and $H$, we show that the squared $L^2 (P)$ error is of order $n^{- \beta/(1+\beta)}$ with high probability. For the settings in which the regression function is bounded, we can use training and validation on the data set to obtain bounds of the same order of $n^{- \beta/(1+\beta)}$. This allows us to select the size of the norm constraint for our Ivanov regularisation without knowing which interpolation space contains the regression function. The response variables in the validation set must have subgaussian errors.

The expectation bounds of order $n^{- \beta / (1 + \beta)}$ for bounded regression functions is optimal in the sense discussed at the end of Section \ref{sIntro}. We use Ivanov regularisation instead of Tikhonov regularisation to control empirical processes over balls in the RKHS. By contrast, the analysis of Tikhonov-regularised estimators usually uses the spectral decomposition of the kernel operator \citep{mendelson2010regularization, steinwart2009optimal}. Analysing the Ivanov-regularised estimator using this decomposition would give a more complete picture of the differences between Ivanov and Tikhonov regularisation for RKHS regression.

It would be useful to extend the lower bound of order $n^{- \beta / (1 + \beta)}$, discussed at the end of Section \ref{sIntro}, to the case in which the regression function lies in an interpolation space between $L^\infty$ and the RKHS. This would show that our high-probability bounds are also of optimal order. However, it is possible that estimation can be performed with a high-probability bound on the squared $L^2 (P)$ error of smaller order.

\appendix

\section{Proof of Expectation Bound for Unbounded Regression Function}

The proofs of all of the bounds in this paper follow the outline in Section \ref{sIvanov}. We first prove Lemma \ref{lIneq}.

{\bf Proof of Lemma \ref{lIneq}} Since $f \in A$, the definition of $\hat f$ gives
\begin{equation*}
\frac{1}{n} \sum_{i = 1}^n (\hat f(X_i) - Y_i)^2 \leq \frac{1}{n} \sum_{i = 1}^n (f (X_i) - Y_i)^2.
\end{equation*}
Expanding
\begin{equation*}
(\hat f(X_i) - Y_i )^2 = \left ( (\hat f(X_i) -  f(X_i)) + (f(X_i) - Y_i) \right )^2,
\end{equation*}
substituting into the above and rearranging gives
\begin{equation*}
\frac{1}{n} \sum_{i = 1}^n (\hat f(X_i) - f(X_i))^2 \leq \frac{2}{n} \sum_{i = 1}^n (Y_i - f(X_i)) (\hat f(X_i) - f(X_i)).
\end{equation*}
Substituting
\begin{equation*}
Y_i - f(X_i) = (Y_i - g(X_i)) + (g(X_i) - f(X_i))
\end{equation*}
into the above and applying the Cauchy--Schwarz inequality to the second term gives
\begin{align*}
\lVert \hat f - f \rVert_{L^2 (P_n)}^2 &\leq \frac{2}{n} \sum_{i = 1}^n (Y_i - g(X_i)) (\hat f(X_i) - f(X_i)) \\
&+ 2 \lVert g - f \rVert_{L^2 (P_n)} \lVert \hat f - f \rVert_{L^2 (P_n)}.
\end{align*}
For constants $a, b \in \real$ and a variable $x \in \real$, we have
\begin{equation*}
x^2 \leq a + 2 b x \implies x^2 \leq 2 a + 4 b^2
\end{equation*}
by completing the square and rearranging. Applying this result to the above inequality proves the lemma. \hfill \BlackBox

The following lemma is useful for bounding the expectation of both of the suprema in Section \ref{sIvanov}.

\begin{lemma} \label{lEmpProc}
Assume \ref{H}. Let the $\eps_i$ be random variables on $(\Omega, \mathcal{F}, \prob)$ such that $\expect(\eps_i \vert X) = 0$ almost surely and $\vary(\eps_i \vert X) \leq \sigma^2$ almost surely for $1 \leq i \leq n$ and $\cov(\eps_i, \eps_j \vert X) = 0$ almost surely for $1 \leq i, j \leq n$ with $i \neq j$. Then
\begin{equation*}
\expect \left ( \sup_{f \in r B_H} \left \lvert \frac{1}{n} \sum_{i = 1}^n \eps_i f(X_i) \right \rvert \right ) \leq \frac{\lVert k \rVert_\infty \sigma r}{n^{1/2}}.
\end{equation*}
\end{lemma}

\begin{proof}
This proof method is due to Remark 6.1 of \citet{sriperumbudur2016optimal}. By the reproducing kernel property and the Cauchy--Schwarz inequality, we have
\begin{align*}
\sup_{f \in r B_H} \left \lvert \frac{1}{n} \sum_{i = 1}^n \eps_i f(X_i) \right \rvert &= \sup_{f \in r B_H} \left \lvert \left \langle \frac{1}{n} \sum_{i = 1}^n \eps_i k_{X_i}, f \right \rangle_H \right \rvert \\
&= r \left \lVert \frac{1}{n} \sum_{i = 1}^n \eps_i k_{X_i} \right \rVert_H \\
&= r \left ( \frac{1}{n^2} \sum_{i, j = 1}^n \eps_i \eps_j k(X_i, X_j) \right )^{1/2}.
\end{align*}
By Jensen's inequality, we have
\begin{align*}
\expect \left ( \sup_{f \in r B_H} \left \lvert \frac{1}{n} \sum_{i = 1}^n \eps_i f(X_i) \right \rvert \; \middle \vert X \right ) &\leq r \left ( \frac{1}{n^2} \sum_{i, j = 1}^n \cov(\eps_i, \eps_j \vert X ) k(X_i, X_j) \right )^{1/2} \\
&\leq r \left ( \frac{\sigma^2}{n^2} \sum_{i = 1}^n k(X_i, X_i) \right )^{1/2}
\end{align*}
almost surely and again, by Jensen's inequality, we have
\begin{equation*}
\expect \left ( \sup_{f \in r B_H} \left \lvert \frac{1}{n} \sum_{i = 1}^n \eps_i f(X_i) \right \rvert \right ) \leq r \left ( \frac{\sigma^2 \lVert k \rVert_\infty^2}{n} \right )^{1/2}.
\end{equation*}
The result follows.
\end{proof}
We bound the distance between $\hat h_r$ and $h_r$ in the $L^2(P_n)$ norm for $r > 0$ and $h_r \in r B_H$.

\begin{lemma} \label{lEmpNorm}
Assume \ref{Y1} and \ref{H}. Let $h_r \in r B_H$. We have
\begin{equation*}
\expect \left (\lVert \hat h_r -  h_r \rVert_{L^2 (P_n)}^2 \right ) \leq \frac{4 \lVert k \rVert_\infty \sigma r}{n^{1/2}} + 4 \lVert h_r - g \rVert_{L^2 (P)}^2.
\end{equation*}
\end{lemma}

\begin{proof}
By Lemma \ref{lIneq} with $A = r B_H$, we have
\begin{equation*}
\lVert \hat h_r -  h_r \rVert_{L^2 (P_n)}^2 \leq \frac{4}{n} \sum_{i = 1}^n (Y_i - g(X_i)) (\hat h_r (X_i) - h_r (X_i)) + 4 \lVert h_r - g \rVert_{L^2 (P_n)}^2.
\end{equation*}
We now bound the expectation of the right-hand side. We have
\begin{equation*}
\expect \left (\lVert h_r - g \rVert_{L^2 (P_n)}^2 \right ) = \lVert h_r - g \rVert_{L^2 (P)}^2.
\end{equation*}
Furthermore,
\begin{equation*}
\expect \left ( \frac{1}{n} \sum_{i = 1}^n (Y_i - g(X_i)) h_r (X_i) \right ) = \expect \left ( \frac{1}{n} \sum_{i = 1}^n \expect(Y_i - g(X_i) \vert X_i) h_r (X_i) \right ) = 0.
\end{equation*}
Finally, by Lemma \ref{lEmpProc} with $\eps_i = Y_i - g(X_i)$, we have
\begin{align*}
\expect \left (\frac{1}{n} \sum_{i = 1}^n (Y_i - g(X_i)) \hat h_r (X_i) \right ) &\leq \expect \left ( \sup_{f \in r B_H} \left \lvert \frac{1}{n} \sum_{i = 1}^n (Y_i - g(X_i)) f(X_i) \right \rvert \right ) \\
&\leq \frac{\lVert k \rVert_\infty \sigma r}{n^{1/2}}.
\end{align*}
The result follows.
\end{proof}
The following lemma is useful for moving the bound on the distance between $\hat h_r$ and $h_r$ from the $L^2(P_n)$ norm to the $L^2 (P)$ norm for $r > 0$ and $h_r \in r B_H$.

\begin{lemma} \label{lSwapNorm}
Assume \ref{H}. We have
\begin{equation*}
\expect \left (\sup_{f \in r B_H} \left \lvert \lVert f \rVert_{L^2 (P_n)}^2 - \lVert f \rVert_{L^2 (P)}^2 \right \rvert \right ) \leq \frac{8 \lVert k \rVert_\infty^2 r^2}{n^{1/2}}.
\end{equation*}
\end{lemma}

\begin{proof}
Let the $\eps_i$ for $1 \leq i \leq n$ be \iid Rademacher random variables on $(\Omega, \mathcal{F}, \prob)$, independent of the $X_i$. Lemma 2.3.1 of \citet{van1996weak} shows
\begin{equation*}
\expect \left (\sup_{f \in r B_H} \left \lvert \frac{1}{n} \sum_{i = 1}^n f (X_i)^2 - \int f^2 dP \right \rvert \right ) \leq 2 \expect \left ( \sup_{f \in r B_H} \left \lvert \frac{1}{n} \sum_{i = 1}^n \eps_i f(X_i)^2 \right \rvert \right )
\end{equation*}
by symmetrisation. Since $\lvert f (X_i) \rvert \leq \lVert k \rVert_\infty r$ for all $f \in r B_H$, we find
\begin{equation*}
\frac{f (X_i)^2}{2 \lVert k \rVert_\infty r}
\end{equation*}
is a contraction vanishing at 0 as a function of $f(X_i)$ for all $1 \leq i \leq n$. By Theorem 3.2.1 of \citet{gine2015mathematical}, we have
\begin{equation*}
\expect \left (\sup_{f \in r B_H} \left \lvert \frac{1}{n} \sum_{i = 1}^n \eps_i \frac{f (X_i)^2}{2 \lVert k \rVert_\infty r} \right \rvert \; \middle \vert X \right ) \leq 2 \expect \left (\sup_{f \in r B_H} \left \lvert \frac{1}{n} \sum_{i = 1}^n \eps_i f(X_i) \right \rvert \; \middle \vert X \right )
\end{equation*}
almost surely. By Lemma \ref{lEmpProc} with $\sigma^2 = 1$, we have
\begin{equation*}
\expect \left ( \sup_{f \in r B_H} \left \lvert \frac{1}{n} \sum_{i = 1}^n \eps_i f(X_i) \right \rvert \right ) \leq \frac{\lVert k \rVert_\infty r}{n^{1/2}}.
\end{equation*}
The result follows.
\end{proof}
We move the bound on the distance between $\hat h_r$ and $h_r$ from the $L^2(P_n)$ norm to the $L^2 (P)$ norm for $r > 0$ and $h_r \in r B_H$.

\begin{corollary} \label{cTrueNorm}
Assume \ref{Y1} and \ref{H}. Let $h_r \in r B_H$. We have
\begin{equation*}
\expect \left (\lVert \hat h_r -  h_r \rVert_{L^2 (P)}^2 \right ) \leq \frac{4 \lVert k \rVert_\infty \sigma r}{n^{1/2}} + \frac{32 \lVert k \rVert_\infty^2 r^2}{n^{1/2}} + 4 \lVert h_r - g \rVert_{L^2 (P)}^2.
\end{equation*}
\end{corollary}

\begin{proof}
By Lemma \ref{lEmpNorm}, we have
\begin{equation*}
\expect \left (\lVert \hat h_r -  h_r \rVert_{L^2 (P_n)}^2 \right ) \leq \frac{4 \lVert k \rVert_\infty \sigma r}{n^{1/2}} + 4 \lVert h_r - g \rVert_{L^2 (P)}^2.
\end{equation*}
Since $\hat h_r -  h_r \in 2 r B_H$, by Lemma \ref{lSwapNorm} we have
\begin{align*}
\expect \left (\lVert \hat h_r -  h_r \rVert_{L^2 (P)}^2 - \lVert \hat h_r -  h_r \rVert_{L^2 (P_n)}^2 \right ) &\leq \expect \left (\sup_{f \in 2 r B_H} \left \lvert \lVert f \rVert_{L^2 (P_n)}^2 - \lVert f \rVert_{L^2 (P)}^2 \right \rvert \right ) \\
&\leq \frac{32 \lVert k \rVert_\infty^2 r^2}{n^{1/2}}.
\end{align*}
The result follows.
\end{proof}
We bound the distance between $\hat h_r$ and $g$ in the $L^2(P)$ norm for $r > 0$ to prove Theorem \ref{tBound}.

{\bf Proof of Theorem \ref{tBound}}
Fix $h_r \in r B_H$. We have
\begin{align*}
\lVert \hat h_r - g \rVert_{L^2 (P)}^2 &\leq \left ( \lVert \hat h_r - h_r \rVert_{L^2 (P)}^2 + \lVert h_r - g \rVert_{L^2 (P)}^2 \right )^2 \\
&\leq 2 \lVert \hat h_r - h_r \rVert_{L^2 (P)}^2 + 2 \lVert h_r - g \rVert_{L^2 (P)}^2.
\end{align*}
By Corollary \ref{cTrueNorm}, we have
\begin{equation*}
\expect \left (\lVert \hat h_r -  h_r \rVert_{L^2 (P)}^2 \right ) \leq \frac{4 \lVert k \rVert_\infty \sigma r}{n^{1/2}} + \frac{32 \lVert k \rVert_\infty^2 r^2}{n^{1/2}} + 4 \lVert h_r - g \rVert_{L^2 (P)}^2.
\end{equation*}
Hence,
\begin{equation*}
\expect \left (\lVert \hat h_r -  g \rVert_{L^2 (P)}^2 \right ) \leq \frac{8 \lVert k \rVert_\infty \sigma r}{n^{1/2}} + \frac{64 \lVert k \rVert_\infty^2 r^2}{n^{1/2}} + 10 \lVert h_r - g \rVert_{L^2 (P)}^2.
\end{equation*}
Taking an infimum over $h_r \in r B_H$ proves the result. \hfill \BlackBox

We assume \ref{g1} to prove Theorem \ref{tInterBound}.

{\bf Proof of Theorem \ref{tInterBound}}
The initial bound follows from Theorem \ref{tBound} and \eqref{eApproxInter}. Based on this bound, setting
\begin{equation*}
r = D_1  \lVert k \rVert_\infty^{- (1-\beta)} B n^{(1-\beta)/4}
\end{equation*}
gives
\begin{equation*}
\expect \left (\lVert \hat h_r -  g \rVert_{L^2 (P)}^2 \right ) \leq \left (64 D_1^2 + 10 D_1^{-2\beta/(1-\beta)} \right ) \lVert k \rVert_\infty^{2 \beta} B^2 n^{- \beta/2} + 8 D_1 \lVert k \rVert_\infty^\beta B \sigma n^{-(1+\beta)/4}.
\end{equation*}
Hence, the next bound follows with
\begin{equation*}
D_2 = 64 D_1^2 + 10 D_1^{-2\beta/(1-\beta)} \text{ and } D_3 = 8 D_1.
\end{equation*} \hfill \BlackBox

\section{Proof of Expectation Bound for Bounded Regression Function}

We can obtain a bound on the distance between $V \hat h_r$ and $V h_r$ in the $L^2(P_n)$ norm for $r > 0$ and $h_r \in r B_H$ from Lemma \ref{lEmpNorm}. The following lemma is useful for moving the bound on the distance between $V \hat h_r$ and $V h_r$ from the $L^2(P_n)$ norm to the $L^2 (P)$ norm.

\begin{lemma} \label{lBddSwapNorm}
Assume \ref{H}. We have
\begin{equation*}
\expect \left (\sup_{f_1, f_2 \in r B_H} \left \lvert \lVert V f_1 - V f_2 \rVert_{L^2 (P_n)}^2 - \lVert V f_2 - V f_2 \rVert_{L^2 (P)}^2 \right \rvert \right ) \leq \frac{64 \lVert k \rVert_\infty C r}{n^{1/2}}.
\end{equation*}
\end{lemma}

\begin{proof}
Let the $\eps_i$ for $1 \leq i \leq n$ be \iid Rademacher random variables on $(\Omega, \mathcal{F}, \prob)$, independent of the $X_i$. Lemma 2.3.1 of \citet{van1996weak} shows
\begin{equation*}
\expect \left (\sup_{f_1, f_2 \in r B_H} \left \lvert \frac{1}{n} \sum_{i = 1}^n (V f_1 (X_i) - V f_2 (X_i))^2 - \int (V f_1 - V f_2)^2 dP \right \rvert \right )
\end{equation*}
is at most
\begin{equation*}
2 \expect \left (\sup_{f_1, f_2 \in r B_H} \left \lvert \frac{1}{n} \sum_{i = 1}^n \eps_i (V f_1 (X_i) - V f_2 (X_i))^2 \right \rvert \right )
\end{equation*}
by symmetrisation. Since $\lvert V f_1 (X_i) - V f_2 (X_i) \rvert \leq 2 C$ for all $f_1, f_2 \in r B_H$, we find
\begin{equation*}
\frac{(V f_1 (X_i) - V f_2 (X_i))^2}{4 C}
\end{equation*}
is a contraction vanishing at 0 as a function of $V f_1 (X_i) - V f_2 (X_i)$ for all $1 \leq i \leq n$. By Theorem 3.2.1 of \citet{gine2015mathematical}, we have
\begin{equation*}
\expect \left (\sup_{f_1, f_2 \in r B_H} \left \lvert \frac{1}{n} \sum_{i = 1}^n \eps_i \frac{(V f_1 (X_i) - V f_2 (X_i))^2}{4 C} \right \rvert \; \middle \vert X \right )
\end{equation*}
is at most
\begin{equation*}
2 \expect \left (\sup_{f_1, f_2 \in r B_H} \left \lvert \frac{1}{n} \sum_{i = 1}^n \eps_i (V f_1 (X_i) - V f_2 (X_i)) \right \rvert \; \middle \vert X \right )
\end{equation*}
almost surely. Therefore,
\begin{equation*}
\expect \left ( \sup_{f_1, f_2 \in r B_H} \left \lvert \frac{1}{n} \sum_{i = 1}^n (V f_1 (X_i) - V f_2 (X_i))^2 - \int (V f_1 - V f_2)^2 dP \right \rvert \right )
\end{equation*}
is at most
\begin{equation*}
16 C \expect \left (\sup_{f_1, f_2 \in r B_H} \left \lvert \frac{1}{n} \sum_{i = 1}^n \eps_i (V f_1 (X_i) - V f_2 (X_i)) \right \rvert \right ) \leq 32 C \expect \left (\sup_{f \in r B_H} \left \lvert \frac{1}{n} \sum_{i = 1}^n \eps_i V f(X_i) \right \rvert \right )
\end{equation*}
by the triangle inequality. Again, by Theorem 3.2.1 of \citet{gine2015mathematical}, we have
\begin{equation*}
\expect \left ( \sup_{f_1, f_2 \in r B_H} \left \lvert \frac{1}{n} \sum_{i = 1}^n (V f_1 (X_i) - V f_2 (X_i))^2 - \int (V f_1 - V f_2)^2 dP \right \rvert \right )
\end{equation*}
is at most
\begin{equation*}
64 C \expect \left (\sup_{f \in r B_H} \left \lvert \frac{1}{n} \sum_{i = 1}^n \eps_i f(X_i) \right \rvert \right )
\end{equation*}
since $V$ is a contraction vanishing at 0. The result follows from Lemma \ref{lEmpProc} with $\sigma^2 = 1$.
\end{proof}
We move the bound on the distance between $V \hat h_r$ and $V h_r$ from the $L^2(P_n)$ norm to the $L^2 (P)$ norm for $r > 0$ and $h_r \in r B_H$.

\begin{corollary} \label{cBddTrueNorm}
Assume \ref{Y1} and \ref{H}. Let $h_r \in r B_H$. We have
\begin{equation*}
\expect \left (\lVert V \hat h_r - V h_r \rVert_{L^2 (P)}^2 \right ) \leq \frac{4 \lVert k \rVert_\infty (16 C + \sigma) r}{n^{1/2}} + 4 \lVert h_r - g \rVert_{L^2 (P)}^2.
\end{equation*}
\end{corollary}

\begin{proof}
By Lemma \ref{lEmpNorm}, we have
\begin{equation*}
\expect \left (\lVert \hat h_r -  h_r \rVert_{L^2 (P_n)}^2 \right ) \leq \frac{4 \lVert k \rVert_\infty \sigma r}{n^{1/2}} + 4 \lVert h_r - g \rVert_{L^2 (P)}^2,
\end{equation*}
so
\begin{equation*}
\expect \left (\lVert V \hat h_r - V h_r \rVert_{L^2 (P_n)}^2 \right ) \leq \frac{4 \lVert k \rVert_\infty \sigma r}{n^{1/2}} + 4 \lVert h_r - g \rVert_{L^2 (P)}^2.
\end{equation*}
Since $\hat h_r, h_r \in r B_H$, by Lemma \ref{lBddSwapNorm} we have
\begin{align*}
&\expect \left (\lVert V \hat h_r -  V h_r \rVert_{L^2 (P)}^2 - \lVert V \hat h_r - V h_r \rVert_{L^2 (P_n)}^2 \right ) \\
\leq \; &\expect \left (\sup_{f_1, f_2 \in r B_H} \left \lvert \lVert V f_1 - V f_2 \rVert_{L^2 (P_n)}^2 - \lVert V f_2 - V f_2 \rVert_{L^2 (P)}^2 \right \rvert \right ) \\
\leq \; &\frac{64 \lVert k \rVert_\infty C r}{n^{1/2}}.
\end{align*}
The result follows.
\end{proof}
We assume \ref{g2} to bound the distance between $V \hat h_r$ and $g$ in the $L^2(P)$ norm for $r > 0$ and prove Theorem \ref{tBddBound}.

{\bf Proof of Theorem \ref{tBddBound}}
Fix $h_r \in r B_H$. We have
\begin{align*}
\lVert V \hat h_r - g \rVert_{L^2 (P)}^2 &\leq \left ( \lVert V \hat h_r - V h_r \rVert_{L^2 (P)}^2 + \lVert V h_r - g \rVert_{L^2 (P)}^2 \right )^2 \\
&\leq 2 \lVert V \hat h_r - V h_r \rVert_{L^2 (P)}^2 + 2 \lVert V h_r - g \rVert_{L^2 (P)}^2 \\
&\leq 2 \lVert V \hat h_r - V h_r \rVert_{L^2 (P)}^2 + 2 \lVert h_r - g \rVert_{L^2 (P)}^2.
\end{align*}
By Corollary \ref{cBddTrueNorm}, we have
\begin{equation*}
\expect \left (\lVert V \hat h_r - V h_r \rVert_{L^2 (P)}^2 \right ) \leq \frac{4 \lVert k \rVert_\infty (16 C + \sigma) r}{n^{1/2}} + 4 \lVert h_r - g \rVert_{L^2 (P)}^2.
\end{equation*}
Hence,
\begin{equation*}
\expect \left (\lVert V \hat h_r -  g \rVert_{L^2 (P)}^2 \right ) \leq \frac{8 \lVert k \rVert_\infty (16 C + \sigma) r}{n^{1/2}} + 10 \lVert h_r - g \rVert_{L^2 (P)}^2.
\end{equation*}
Taking an infimum over $h_r \in r B_H$ proves the result. \hfill \BlackBox

We assume \ref{g1} to prove Theorem \ref{tBddInterBound}.

{\bf Proof of Theorem \ref{tBddInterBound}}
The initial bound follows from Theorem \ref{tBddBound} and \eqref{eApproxInter}. Based on this bound, setting
\begin{equation*}
r = D_1 \lVert k \rVert_\infty^{- (1-\beta)/(1+\beta)} B^{2/(1+\beta)} (16 C + \sigma)^{- (1-\beta)/(1+\beta)} n^{(1-\beta)/(2 (1+\beta))}
\end{equation*}
gives
\begin{equation*}
\expect \left (\lVert V \hat h_r -  g \rVert_{L^2 (P)}^2 \right )
\end{equation*}
is at most
\begin{equation*}
\left (8 D_1 + 10 D_1^{-2\beta/(1-\beta)} \right ) \lVert k \rVert_\infty^{2 \beta/(1+\beta)} B^{2/(1+\beta)} (16 C + \sigma)^{2 \beta/(1+\beta)} n^{- \beta/(1+\beta)}.
\end{equation*}
Hence, the next bound follows with
\begin{equation*}
D_2 = 8 D_1 + 10 D_1^{-2\beta/(1-\beta)}.
\end{equation*} \hfill \BlackBox

\section{Proof of Expectation Bound for Validation}

We need to introduce some definitions for stochastic processes. A stochastic process $W$ on $(\Omega, \mathcal{F})$ indexed by a metric space $(M, d)$ is $d^2$-subgaussian if it is centred and $W(s) - W(t)$ is $d(s, t)^2$-subgaussian for all $s, t \in M$. $W$ is separable if there exists a countable set $M_0 \subs M$ such that the following holds for all $\omega \in \Omega_0$, where $\prob(\Omega_0) = 1$. For all $s \in M$ and $\eps > 0$, $W(s)$ is in the closure of $\{ W(t) : t \in M_0, d(s, t) \leq \eps  \}$.

We also need to introduce the concept of covering numbers for the next result. The covering number $N(M, d, \eps)$ is the minimum number of $d$-balls of size $\eps > 0$ needed to cover $M$.

The following lemma is useful for bounding the expectation of both of the suprema in Section \ref{sIvanov}.

\begin{lemma} \label{lGaussProc}
Assume \ref{H}. Let the $\eps_i$ be random variables on $(\Omega, \mathcal{F}, \prob)$ such that $(\tilde X_i, \eps_i)$ is \iid for $1 \leq i \leq \tilde n$ and let $\eps_i$ be $\tilde \sigma^2$-subgaussian given $\tilde X_i$. Let $r_0 \in R$, $f_0 = V \hat h_{r_0}$ and
\begin{equation*}
W(f) = \frac{1}{\tilde n} \sum_{i = 1}^{\tilde n} \eps_i (f(\tilde X_i) -  f_0 (\tilde X_i))
\end{equation*}
for $f \in F$. Then $W$ is $\tilde \sigma^2 \lVert \cdot \rVert_\infty^2/\tilde n$-subgaussian given $\tilde X$ and separable on $(F, \tilde \sigma \lVert \cdot \rVert_\infty/\tilde n^{1/2})$. Furthermore,
\begin{equation*}
\expect \left (\sup_{f \in F} \lvert W(f) \rvert \right ) \leq \frac{4 C \tilde \sigma}{\tilde n^{1/2}} \left ( \left ( 2 \log \left (2 + \frac{\lVert k \rVert_\infty^2 \rho^2}{C^2} \right ) \right )^{1/2} + \pi^{1/2} \right ).
\end{equation*}
\end{lemma}

\begin{proof}
Let $W_i (f) = \eps_i (f(\tilde X_i) -  f_0 (\tilde X_i))$ for $1 \leq i \leq \tilde n$ and $f \in F$. Note that the $W_i$ are independent and centred. Since $W_i (f_1) - W_i (f_2)$ is $\tilde \sigma^2 \lVert f_1 - f_2 \rVert_\infty^2$-subgaussian given $\tilde X_i$ for $1 \leq i \leq \tilde n$ and $f_1, f_2 \in F$, the process $W$ is $\tilde \sigma^2 \lVert \cdot \rVert_\infty^2/\tilde n$-subgaussian given $\tilde X$. By Lemma \ref{lContEst}, we have that $(F, \tilde \sigma \lVert \cdot \rVert_\infty/\tilde n^{1/2})$ is separable. Hence, $W$ is separable on $(F, \tilde \sigma \lVert \cdot \rVert_\infty/\tilde n^{1/2})$ since it is continuous. The diameter of $(F, \tilde \sigma \lVert \cdot \rVert_\infty/\tilde n^{1/2})$ is
\begin{equation*}
D = \sup_{f_1, f_2 \in F} \tilde \sigma \lVert f_1 - f_2 \rVert_\infty/\tilde n^{1/2} \leq 2 C \tilde \sigma/\tilde n^{1/2}.
\end{equation*}
We have
\begin{align*}
\int_0^\infty ( \log(N(F, \tilde \sigma \lVert \cdot \rVert_\infty/\tilde n^{1/2}, \eps)))^{1/2} d\eps &= \int_0^\infty ( \log(N(F, \lVert \cdot \rVert_\infty, \tilde n^{1/2} \eps / \tilde \sigma)))^{1/2} d\eps \\
&= \frac{\tilde \sigma}{\tilde n^{1/2}} \int_0^\infty ( \log(N(F, \lVert \cdot \rVert_\infty, u)))^{1/2} du.
\end{align*}
This is finite by Lemma \ref{lCoverInt}. Hence, by Theorem 2.3.7 of \citet{gine2015mathematical} and Lemma \ref{lCoverInt}, we have
\begin{equation*}
\expect \left (\sup_{f \in F} \lvert W(f) \rvert \middle \vert \tilde X, X, Y \right )
\end{equation*}
is at most
\begin{align*}
&\expect(\lvert W(f_0) \rvert \; \vert \tilde X, X, Y) + 2^{5/2} \int_0^{C \tilde \sigma/\tilde n^{1/2}} ( \log(2 N(F, \tilde \sigma \lVert \cdot \rVert_\infty/\tilde n^{1/2}, \eps)))^{1/2} d\eps \\
= \; &2^{5/2} \int_0^{C \tilde \sigma/\tilde n^{1/2}} ( \log(2 N(F, \lVert \cdot \rVert_\infty, \tilde n^{1/2} \eps / \tilde \sigma)))^{1/2} d\eps \\
= \; &\frac{2^{5/2} \tilde \sigma}{\tilde n^{1/2}} \int_0^C ( \log(2 N(F, \lVert \cdot \rVert_\infty, u)))^{1/2} du \\
\leq \; &\frac{2^{5/2} \tilde \sigma}{\tilde n^{1/2}} \left ( \left ( \log \left (2 + \frac{\lVert k \rVert_\infty^2 \rho^2}{C^2} \right ) \right )^{1/2} C + \left ( \frac{\pi}{2} \right )^{1/2} C \right ) \\
= \; &\frac{4 C \tilde \sigma}{\tilde n^{1/2}} \left ( \left ( 2 \log \left (2 + \frac{\lVert k \rVert_\infty^2 \rho^2}{C^2} \right ) \right )^{1/2} + \pi^{1/2} \right )
\end{align*}
almost surely, noting $W(f_0) = 0$. The result follows.
\end{proof}
We bound the distance between $V \hat h_{\hat r}$ and $V \hat h_{r_0}$ in the $L^2(\tilde P_{\tilde n})$ norm for $r_0 \in R$.

\begin{lemma} \label{lValEmpNorm}
Assume \ref{H} and \ref{tY}. Let $r_0 \in R$. We have
\begin{equation*}
\expect \left (\lVert V \hat h_{\hat r} -  V \hat h_{r_0} \rVert_{L^2 (\tilde P_{\tilde n})}^2 \right )
\end{equation*}
is at most
\begin{equation*}
\frac{16 C \tilde \sigma}{\tilde n^{1/2}} \left ( \left ( 2 \log \left (2 + \frac{\lVert k \rVert_\infty^2 \rho^2}{C^2} \right ) \right )^{1/2} + \pi^{1/2} \right ) + 4 \expect \left ( \lVert V \hat h_{r_0} - g \rVert_{L^2 (P)}^2 \right ).
\end{equation*}
\end{lemma}

\begin{proof}
By Lemma \ref{lIneq} with $A = F$ and $n$, $X$, $Y$ and $P_n$ replaced by $\tilde n$, $\tilde X$, $\tilde Y$ and $\tilde P_{\tilde n}$, we have
\begin{equation*}
\lVert V \hat h_{\hat r} -  V \hat h_{r_0} \rVert_{L^2 (\tilde P_{\tilde n})}^2 \leq \frac{4}{\tilde n} \sum_{i = 1}^{\tilde n} (\tilde Y_i - g(\tilde X_i)) (V \hat h_{\hat r} (\tilde X_i) -  V \hat h_{r_0} (\tilde X_i)) + 4 \lVert V \hat h_{r_0} - g \rVert_{L^2 (\tilde P_{\tilde n})}^2.
\end{equation*}
We now bound the expectation of the right-hand side. We have
\begin{equation*}
\expect \left ( \lVert V \hat h_{r_0} - g \rVert_{L^2 (\tilde P_{\tilde n})}^2 \right ) = \expect \left ( \lVert V \hat h_{r_0} - g \rVert_{L^2 (P)}^2 \right ).
\end{equation*}
Let $f_0 = V \hat h_{r_0}$. By Lemma \ref{lGaussProc} with $\eps_i = Y_i - g(X_i)$, we have
\begin{align*}
&\expect \left ( \frac{1}{\tilde n} \sum_{i = 1}^{\tilde n} (\tilde Y_i - g(\tilde X_i)) (V \hat h_{\hat r} (\tilde X_i) -  V \hat h_{r_0} (\tilde X_i)) \right ) \\
\leq \; &\expect \left ( \sup_{f \in F} \left \lvert \frac{1}{\tilde n} \sum_{i = 1}^{\tilde n} (\tilde Y_i - g(\tilde X_i)) (f(\tilde X_i) -  f_0 (\tilde X_i)) \right \rvert \right ) \\
\leq \; &\frac{4 C \tilde \sigma}{\tilde n^{1/2}} \left ( \left ( 2 \log \left (2 + \frac{\lVert k \rVert_\infty^2 \rho^2}{C^2} \right ) \right )^{1/2} + \pi^{1/2} \right ).
\end{align*}
The result follows.
\end{proof}
The following lemma is useful for moving the bound on the distance between $V \hat h_{\hat r}$ and $V \hat h_{r_0}$ from the $L^2(\tilde P_{\tilde n})$ norm to the $L^2 (P)$ norm for $r_0 \in R$.

\begin{lemma} \label{lValSwapNorm}
Assume \ref{H}. Let $r_0 \in R$ and $f_0 = V \hat h_{r_0}$. We have
\begin{equation*}
\expect \left (\sup_{f \in F} \left \lvert \lVert f -  f_0 \rVert_{L^2 (\tilde P_{\tilde n})}^2 - \lVert f -  f_0 \rVert_{L^2 (P)}^2 \right \rvert \right ) \leq \frac{64 C^2}{\tilde n^{1/2}} \left ( \left ( 2 \log \left (2 + \frac{\lVert k \rVert_\infty^2 \rho^2}{C^2} \right ) \right )^{1/2} + \pi^{1/2} \right ).
\end{equation*}
\end{lemma}

\begin{proof}
Let the $\eps_i$ be \iid Rademacher random variables on $(\Omega, \mathcal{F}, \prob)$ for $1 \leq i \leq \tilde n$, independent of $\tilde X$, $X$ and $Y$. Lemma 2.3.1 of \citet{van1996weak} shows
\begin{equation*}
\expect \left (\sup_{f \in F} \left \lvert \frac{1}{\tilde n} \sum_{i = 1}^{\tilde n} (f(\tilde X_i) - f_0 (\tilde X_i))^2 - \int (f - f_0)^2 dP \right \rvert \; \middle \vert X, Y \right )
\end{equation*}
is at most
\begin{equation*}
2 \expect \left (\sup_{f \in F} \left \lvert \frac{1}{\tilde n} \sum_{i = 1}^{\tilde n} \eps_i (f(\tilde X_i) - f_0 (\tilde X_i))^2 \right \rvert \; \middle \vert X, Y \right )
\end{equation*}
almost surely by symmetrisation. Since $\lvert f(\tilde X_i) - f_0 (\tilde X_i) \rvert \leq 2 C$ for all $f \in F$, we find
\begin{equation*}
\frac{(f(\tilde X_i) - f_0 (\tilde X_i))^2}{4 C}
\end{equation*}
is a contraction vanishing at 0 as a function of $f(\tilde X_i) - f_0 (\tilde X_i)$ for all $1 \leq i \leq \tilde n$. By Theorem 3.2.1 of \citet{gine2015mathematical}, we have
\begin{equation*}
\expect \left (\sup_{f \in F} \left \lvert \frac{1}{\tilde n} \sum_{i = 1}^{\tilde n} \eps_i \frac{(f(\tilde X_i) - f_0 (\tilde X_i))^2}{4 C} \right \rvert \; \middle \vert \tilde X, X, Y \right )
\end{equation*}
is at most
\begin{equation*}
2 \expect \left (\sup_{f \in F} \left \lvert \frac{1}{\tilde n} \sum_{i = 1}^{\tilde n} \eps_i (f(\tilde X_i) - f_0 (\tilde X_i)) \right \rvert \; \middle \vert \tilde X, X, Y \right )
\end{equation*}
almost surely. Therefore,
\begin{equation*}
\expect \left (\sup_{f \in F} \left \lvert \frac{1}{\tilde n} \sum_{i = 1}^{\tilde n} (f(\tilde X_i) - f_0 (\tilde X_i))^2 - \int (f - f_0)^2 dP \right \rvert \; \middle \vert X, Y \right )
\end{equation*}
is at most
\begin{equation*}
16 C \expect \left (\sup_{f \in F} \left \lvert \frac{1}{\tilde n} \sum_{i = 1}^{\tilde n} \eps_i (f(\tilde X_i) - f_0 (\tilde X_i)) \right \rvert \; \middle \vert X, Y \right )
\end{equation*}
almost surely. The result follows from Lemma \ref{lGaussProc} with $\tilde \sigma^2 = 1$.
\end{proof}
We move the bound on the distance between $V \hat h_{\hat r}$ and $V \hat h_{r_0}$ from the $L^2(\tilde P_{\tilde n})$ norm to the $L^2 (P)$ norm for $r_0 \in R$.

\begin{corollary} \label{cValTrueNorm}
Assume \ref{H} and \ref{tY}. Let $r_0 \in R$. We have
\begin{equation*}
\expect \left (\lVert V \hat h_{\hat r} -  V \hat h_{r_0} \rVert_{L^2 (P)}^2 \right )
\end{equation*}
is at most
\begin{equation*}
\frac{16 C (4 C + \tilde \sigma)}{\tilde n^{1/2}} \left ( \left ( 2 \log \left (2 + \frac{\lVert k \rVert_\infty^2 \rho^2}{C^2} \right ) \right )^{1/2} + \pi^{1/2} \right ) + 4 \expect \left ( \lVert V \hat h_{r_0} - g \rVert_{L^2 (P)}^2 \right ).
\end{equation*}
\end{corollary}

\begin{proof}
By Lemma \ref{lValEmpNorm}, we have
\begin{equation*}
\expect \left (\lVert V \hat h_{\hat r} -  V \hat h_{r_0} \rVert_{L^2 (\tilde P_{\tilde n})}^2 \right )
\end{equation*}
is at most
\begin{equation*}
\frac{16 C \tilde \sigma}{\tilde n^{1/2}} \left ( \left ( 2 \log \left (2 + \frac{\lVert k \rVert_\infty^2 \rho^2}{C^2} \right ) \right )^{1/2} + \pi^{1/2} \right ) + 4 \expect \left ( \lVert V \hat h_{r_0} - g \rVert_{L^2 (P)}^2 \right ).
\end{equation*}
Let $f_0 = V \hat h_{r_0}$. Since $\hat h_{\hat r} \in F$, by Lemma \ref{lValSwapNorm} we have
\begin{align*}
&\expect \left (\lVert V \hat h_{\hat r} -  V \hat h_{r_0} \rVert_{L^2 (P)}^2 - \lVert V \hat h_{\hat r} -  V \hat h_{r_0} \rVert_{L^2 (\tilde P_{\tilde n})}^2 \right ) \\
\leq \; &\expect \left (\sup_{f \in F} \left \lvert \lVert f -  f_0 \rVert_{L^2 (\tilde P_{\tilde n})}^2 - \lVert f -  f_0 \rVert_{L^2 (P)}^2 \right \rvert \right ) \\
\leq \; &\frac{64 C^2}{\tilde n^{1/2}} \left ( \left ( 2 \log \left (2 + \frac{\lVert k \rVert_\infty^2 \rho^2}{C^2} \right ) \right )^{1/2} + \pi^{1/2} \right ).
\end{align*}
The result follows.
\end{proof}
We bound the distance between $V \hat h_{\hat r}$ and $g$ in the $L^2(P)$ norm to prove Theorem \ref{tValBound}.

{\bf Proof of Theorem \ref{tValBound}}
We have
\begin{align*}
\lVert V \hat h_{\hat r} - g \rVert_{L^2 (P)}^2 &\leq \left ( \lVert V \hat h_{\hat r} - V h_{r_0} \rVert_{L^2 (P)}^2 + \lVert V h_{r_0} - g \rVert_{L^2 (P)}^2 \right )^2 \\
&\leq 2 \lVert V \hat h_{\hat r} - V h_{r_0} \rVert_{L^2 (P)}^2 + 2 \lVert V h_{r_0} - g \rVert_{L^2 (P)}^2.
\end{align*}
By Corollary \ref{cValTrueNorm}, we have
\begin{equation*}
\expect \left (\lVert V \hat h_{\hat r} -  V \hat h_{r_0} \rVert_{L^2 (P)}^2 \right )
\end{equation*}
is at most
\begin{equation*}
\frac{16 C (4 C + \tilde \sigma)}{\tilde n^{1/2}} \left ( \left ( 2 \log \left (2 + \frac{\lVert k \rVert_\infty^2 \rho^2}{C^2} \right ) \right )^{1/2} + \pi^{1/2} \right ) + 4 \expect \left ( \lVert V \hat h_{r_0} - g \rVert_{L^2 (P)}^2 \right ).
\end{equation*}
Hence,
\begin{equation*}
\expect \left (\lVert V \hat h_{\hat r} -  g \rVert_{L^2 (P)}^2 \right )
\end{equation*}
is at most
\begin{equation*}
\frac{32 C (4 C + \tilde \sigma)}{\tilde n^{1/2}} \left ( \left ( 2 \log \left (2 + \frac{\lVert k \rVert_\infty^2 \rho^2}{C^2} \right ) \right )^{1/2} + \pi^{1/2} \right ) + 10 \expect \left ( \lVert V \hat h_{r_0} - g \rVert_{L^2 (P)}^2 \right ).
\end{equation*} \hfill \BlackBox

We assume the conditions of Theorem \ref{tBddInterBound} to prove Theorem \ref{tValInterBound}.

{\bf Proof of Theorem \ref{tValInterBound}}
If we assume \ref{R1}, then $r_0 = a n^{(1-\beta)/(2 (1+\beta))} \in R$ and
\begin{equation*}
\expect \left (\lVert V \hat h_{r_0} -  g \rVert_{L^2 (P)}^2 \right ) \leq \frac{8 \lVert k \rVert_\infty (16 C + \sigma) a n^{(1-\beta)/(2 (1+\beta))}}{n^{1/2}} + \frac{10 B^{2/(1-\beta)}}{a^{ 2 \beta/(1-\beta)} n^{\beta/(1+\beta)}}
\end{equation*}
by Theorem \ref{tBddInterBound}. If we assume \ref{R2}, then there is at least one $r_0 \in R$ such that
\begin{equation*}
a n^{(1-\beta)/(2 (1+\beta))} \leq r_0 < a n^{(1-\beta)/(2 (1+\beta))} + b
\end{equation*}
and
\begin{align*}
\expect \left (\lVert V \hat h_{r_0} -  g \rVert_{L^2 (P)}^2 \right ) &\leq \frac{8 \lVert k \rVert_\infty (16 C + \sigma) r_0}{n^{1/2}} + \frac{10 B^{2/(1-\beta)}}{r_0^{ 2 \beta/(1-\beta)}} \\
&\leq \frac{8 \lVert k \rVert_\infty (16 C + \sigma) \left (a n^{(1-\beta)/(2 (1+\beta))} + b \right )}{n^{1/2}} + \frac{10 B^{2/(1-\beta)}}{a^{ 2 \beta/(1-\beta)} n^{\beta/(1+\beta)}}
\end{align*}
by Theorem \ref{tBddInterBound}. In either case,
\begin{equation*}
\expect \left (\lVert V \hat h_{r_0} -  g \rVert_{L^2 (P)}^2 \right ) \leq D_2 n^{-\beta/(1+\beta)}
\end{equation*}
for some constant $D_2 > 0$ not depending on $n$ or $\tilde n$. By Theorem \ref{tValBound}, we have
\begin{equation*}
\expect \left (\lVert V \hat h_{\hat r} -  g \rVert_{L^2 (P)}^2 \right ) \leq D_3 \log(n)^{1/2} \tilde n^{-1/2} + 10 D_2 n^{-\beta/(1+\beta)}
\end{equation*}
for some constant $D_3 > 0$ not depending on $n$ or $\tilde n$. Since $\tilde n$ increases at least linearly in $n$, there exists some constant $D_4 > 0$ such that $\tilde n \geq D_4 n$. We then have
\begin{align*}
\expect \left (\lVert V \hat h_{\hat r} -  g \rVert_{L^2 (P)}^2 \right ) &\leq D_4^{-1/2} D_3 \log(n)^{1/2} n^{-1/2} + 10 D_2 n^{-\beta/(1+\beta)} \\
&\leq D_1 n^{-\beta/(1+\beta)}
\end{align*}
for some constant $D_1 > 0$ not depending on $n$ or $\tilde n$. \hfill \BlackBox

\section{Proof of High-Probability Bound for Bounded Regression Function}

We bound the distance between $\hat h_r$ and $h_r$ in the $L^2(P_n)$ norm for $r > 0$ and $h_r \in r B_H$.

\begin{lemma} \label{lProbBddEmpNorm}
Assume \ref{Y2} and \ref{H}. Let $r > 0$, $h_r \in r B_H$ and $t \geq 1$. With probability at least $1 - 2 e^{-t}$, we have
\begin{equation*}
\lVert \hat h_r -  h_r \rVert_{L^2 (P_n)}^2 \leq \frac{20 \lVert k \rVert_\infty \sigma r t^{1/2}}{n^{1/2}} + 4 \lVert h_r - g \rVert_\infty^2.
\end{equation*}
\end{lemma}

\begin{proof}
By Lemma \ref{lIneq} with $A = r B_H$, we have
\begin{equation*}
\lVert \hat h_r -  h_r \rVert_{L^2 (P_n)}^2 \leq \frac{4}{n} \sum_{i = 1}^n (Y_i - g(X_i)) (\hat h_r (X_i) - h_r (X_i)) + 4 \lVert h_r - g \rVert_{L^2 (P_n)}^2.
\end{equation*}
We now bound the right-hand side. We have
\begin{equation*}
\lVert h_r - g \rVert_{L^2 (P_n)}^2 \leq \lVert h_r - g \rVert_\infty^2.
\end{equation*}
Furthermore,
\begin{equation*}
- \frac{1}{n} \sum_{i = 1}^n (Y_i - g(X_i)) h_r (X_i)
\end{equation*}
is subgaussian given $X$ with parameter
\begin{equation*}
\frac{1}{n^2} \sum_{i = 1}^n \sigma^2 h_r (X_i)^2 \leq \frac{\lVert k \rVert_\infty^2 \sigma^2 r^2}{n}.
\end{equation*}
So we have
\begin{equation*}
- \frac{1}{n} \sum_{i = 1}^n (Y_i - g(X_i)) h_r (X_i) \leq \frac{\lVert k \rVert_\infty \sigma r (2 t)^{1/2}}{n^{1/2}} \leq \frac{2 \lVert k \rVert_\infty \sigma r t^{1/2}}{n^{1/2}}
\end{equation*}
with probability at least $1 - e^{-t}$ by Chernoff bounding. Finally, we have
\begin{align*}
\frac{1}{n} \sum_{i = 1}^n (Y_i - g(X_i)) \hat h_r (X_i) &\leq \sup_{f \in r B_H} \left \lvert \frac{1}{n} \sum_{i = 1}^n (Y_i - g(X_i)) f(X_i) \right \rvert \\
&= \sup_{f \in r B_H} \left \lvert \left \langle \frac{1}{n} \sum_{i = 1}^n (Y_i - g(X_i)) k_{X_i}, f \right \rangle_H \right \rvert \\
&= r \left \lVert \frac{1}{n} \sum_{i = 1}^n (Y_i - g(X_i)) k_{X_i} \right \rVert_H \\
&= r \left ( \frac{1}{n^2} \sum_{i, j = 1}^n (Y_i - g(X_i)) (Y_j - g(X_j)) k(X_i, X_j) \right )^{1/2}
\end{align*}
by the reproducing kernel property and the Cauchy--Schwarz inequality. Let $K$ be the $n \times n$ matrix with $K_{i, j} = k(X_i, X_j)$ and let $\eps$ be the vector of the $Y_i - g(X_i)$. Then
\begin{equation*}
\frac{1}{n^2} \sum_{i, j = 1}^n (Y_i - g(X_i)) (Y_j - g(X_j)) k(X_i, X_j) = \eps^{\trans} (n^{-2} K) \eps.
\end{equation*}
Furthermore, since $k$ is a measurable function on $(S \times S, \mathcal{S} \otimes \mathcal{S})$, we have that $n^{-2} K$ is an $(\real^{n \times n}, \bor(\real^{n \times n}))$-valued measurable matrix on $(\Omega, \mathcal{F})$ and non-negative-definite. Let $a_i$ for $1 \leq i \leq n$ be the eigenvalues of $n^{-2} K$. Then
\begin{equation*}
\max_i a_i \leq \tr(n^{-2} K) \leq n^{-1} \lVert k \rVert_\infty^2
\end{equation*}
and
\begin{equation*}
\tr((n^{-2} K)^2) = \lVert a \rVert_2^2 \leq \lVert a \rVert_1^2 \leq n^{-2} \lVert k \rVert_\infty^4.
\end{equation*}
Therefore, by Lemma \ref{lSubgaussProb} with $M = n^{-2} K$, we have
\begin{equation*}
\eps^{\trans} (n^{-2} K) \eps \leq \lVert k \rVert_\infty^2 \sigma^2 n^{-1} (1 + 2 t + 2 (t^2 + t)^{1/2})
\end{equation*}
and
\begin{equation*}
\frac{1}{n} \sum_{i = 1}^n (Y_i - g(X_i)) \hat h_r (X_i) \leq \frac{3 \lVert k \rVert_\infty \sigma r t^{1/2}}{n^{1/2}}
\end{equation*}
with probability at least $1 - e^{-t}$. The result follows.
\end{proof}
The following lemma is useful for bounding the supremum in \eqref{eSup2}.

\begin{lemma} \label{lTal}
Let $D > 0$ and $A \subs L^\infty$ be separable with $\lVert f \rVert_\infty \leq D$ for all $f \in A$. Let
\begin{equation*}
Z = \sup_{f \in A} \left \lvert \lVert f \rVert_{L^2 (P_n)}^2 - \lVert f \rVert_{L^2 (P)}^2 \right \rvert.
\end{equation*}
Then, for $t > 0$, we have
\begin{equation*}
Z \leq \expect(Z) + \left ( \frac{2 D^4 t}{n} + \frac{4 D^2 \expect(Z) t}{n} \right )^{1/2} + \frac{2 D^2 t}{3 n}
\end{equation*}
with probability at least $1 - e^{-t}$.
\end{lemma}

\begin{proof}
We have
\begin{equation*}
Z = \sup_{f \in A} \left \lvert\sum_{i = 1}^n n^{-1} \left ( f(X_i)^2 - \lVert f \rVert_{L^2 (P)}^2 \right ) \right \rvert
\end{equation*}
and
\begin{align*}
\expect \left ( n^{-1} \left ( f(X_i)^2 - \lVert f \rVert_{L^2 (P)}^2 \right ) \right ) &= 0, \\
n^{-1} \left \lvert f(X_i)^2 - \lVert f \rVert_{L^2 (P)}^2 \right \rvert &\leq \frac{D^2}{n}, \\
\expect \left ( n^{-2} \left ( f(X_i)^2 - \lVert f \rVert_{L^2 (P)}^2 \right )^2 \right ) &\leq \frac{D^4}{n^2}
\end{align*}
for all $1 \leq i \leq n$ and $f \in A$. Furthermore, $A$ is separable, so $Z$ is a random variable on $(\Omega, \mathcal{F})$ and we can use Talagrand's inequality \citep[Theorem A.9.1 of][]{steinwart2008support} to show
\begin{equation*}
Z > \expect(Z) + \left (2 t \left ( \frac{D^4}{n} + \frac{2 D^2 \expect(Z)}{n} \right ) \right )^{1/2} + \frac{2 t D^2}{3 n}
\end{equation*}
with probability at most $e^{-t}$. The result follows.
\end{proof}
The following lemma is useful for moving the bound on the distance between $V \hat h_r$ and $V h_r$ from the $L^2(P_n)$ norm to the $L^2 (P)$ norm for $r > 0$ and $h_r \in r B_H$.

\begin{lemma} \label{lProbBddSwapNorm}
Assume \ref{H}. Let $r > 0$ and $t \geq 1$. With probability at least $1 - e^{-t}$, we have
\begin{equation*}
\sup_{f_1, f_2 \in r B_H} \left \lvert \lVert V f_1 - V f_2 \rVert_{L^2 (P_n)}^2 - \lVert V f_1 - V f_2 \rVert_{L^2 (P)}^2 \right \rvert
\end{equation*}
is at most
\begin{equation*}
\frac{8 \left (C^2 + 4 \lVert k \rVert_\infty^{1/2} C^{3/2} r^{1/2} + 8 \lVert k \rVert_\infty C r \right ) t^{1/2}}{n^{1/2}} + \frac{8 C^2 t}{3 n}.
\end{equation*}
\end{lemma}

\begin{proof}
Let $A = \{ V f_1 - V f_2 : f_1, f_2 \in r B_H \}$ and
\begin{equation*}
Z = \sup_{f_1, f_2 \in r B_H} \left \lvert \lVert V f_1 - V f_2 \rVert_{L^2 (P_n)}^2 - \lVert V f_1 - V f_2 \rVert_{L^2 (P)}^2 \right \rvert.
\end{equation*}
Then $A \subs L^\infty$ is separable because $H$ is separable and has a bounded kernel $k$. Furthermore, $\lVert V f_1 - V f_2 \rVert_\infty \leq 2 C$ for all $f_1, f_2 \in r B_H$. By Lemma \ref{lTal}, we have
\begin{equation*}
Z \leq \expect(Z) + \left ( \frac{32 C^4 t}{n} + \frac{16 C^2 \expect(Z) t}{n} \right )^{1/2} + \frac{8 C^2 t}{3 n}
\end{equation*}
with probability at least $1 - e^{-t}$. By Lemma \ref{lBddSwapNorm}, we have
\begin{equation*}
\expect(Z) \leq \frac{64 \lVert k \rVert_\infty C r}{n^{1/2}}.
\end{equation*}
The result follows.
\end{proof}
We move the bound on the distance between $V \hat h_r$ and $V h_r$ from the $L^2(P_n)$ norm to the $L^2 (P)$ norm for $r > 0$ and $h_r \in r B_H$.

\begin{corollary} \label{cProbBddTrueNorm}
Assume \ref{Y2} and \ref{H}. Let $r > 0$, $h_r \in r B_H$ and $t \geq 1$. With probability at least $1 - 3 e^{-t}$, we have
\begin{equation*}
\lVert V \hat h_r - V h_r \rVert_{L^2 (P)}^2
\end{equation*}
is at most
\begin{equation*}
\frac{4 \left (2 C^2 + 8 \lVert k \rVert_\infty^{1/2} C^{3/2} r^{1/2} + \lVert k \rVert_\infty (16 C + 5 \sigma) r \right ) t^{1/2}}{n^{1/2}} + \frac{8 C^2 t}{3 n} + 4 \lVert h_r - g \rVert_\infty^2.
\end{equation*}
\end{corollary}

\begin{proof}
By Lemma \ref{lProbBddEmpNorm}, we have
\begin{equation*}
\lVert \hat h_r -  h_r \rVert_{L^2 (P_n)}^2 \leq \frac{20 \lVert k \rVert_\infty \sigma r t^{1/2}}{n^{1/2}} + 4 \lVert h_r - g \rVert_\infty^2.
\end{equation*}
with probability at least $1 - 2 e^{-t}$, so
\begin{equation*}
\lVert V \hat h_r -  V h_r \rVert_{L^2 (P_n)}^2 \leq \frac{20 \lVert k \rVert_\infty \sigma r t^{1/2}}{n^{1/2}} + 4 \lVert h_r - g \rVert_\infty^2.
\end{equation*}
Since $\hat h_r, h_r \in r B_H$, by Lemma \ref{lProbBddSwapNorm} we have
\begin{align*}
&\lVert V \hat h_r -  V h_r \rVert_{L^2 (P)}^2 - \lVert V \hat h_r - V h_r \rVert_{L^2 (P_n)}^2 \\
\leq \; &\sup_{f_1, f_2 \in r B_H} \left \lvert \lVert V f_1 - V f_2 \rVert_{L^2 (P_n)}^2 - \lVert V f_2 - V f_2 \rVert_{L^2 (P)}^2 \right \rvert \\
\leq \; &\frac{8 \left (C^2 + 4 \lVert k \rVert_\infty^{1/2} C^{3/2} r^{1/2} + 8 \lVert k \rVert_\infty C r \right ) t^{1/2}}{n^{1/2}} + \frac{8 C^2 t}{3 n}
\end{align*}
with probability at least $1 - e^{-t}$. The result follows.
\end{proof}
We assume \ref{g2} to bound the distance between $V \hat h_r$ and $g$ in the $L^2(P)$ norm for $r > 0$ and prove Theorem \ref{tProbBddBound}.

{\bf Proof of Theorem \ref{tProbBddBound}}
Fix $h_r \in r B_H$. We have
\begin{align*}
\lVert V \hat h_r - g \rVert_{L^2 (P)}^2 &\leq \left ( \lVert V \hat h_r - V h_r \rVert_{L^2 (P)}^2 + \lVert V h_r - g \rVert_{L^2 (P)}^2 \right )^2 \\
&\leq 2 \lVert V \hat h_r - V h_r \rVert_{L^2 (P)}^2 + 2 \lVert V h_r - g \rVert_{L^2 (P)}^2 \\
&\leq 2 \lVert V \hat h_r - V h_r \rVert_{L^2 (P)}^2 + 2 \lVert h_r - g \rVert_{L^2 (P)}^2.
\end{align*}
By Corollary \ref{cProbBddTrueNorm}, we have
\begin{equation*}
\lVert V \hat h_r - V h_r \rVert_{L^2 (P)}^2
\end{equation*}
is at most
\begin{equation*}
\frac{4 \left (2 C^2 + 8 \lVert k \rVert_\infty^{1/2} C^{3/2} r^{1/2} + \lVert k \rVert_\infty (16 C + 5 \sigma) r \right ) t^{1/2}}{n^{1/2}} + \frac{8 C^2 t}{3 n} + 4 \lVert h_r - g \rVert_\infty^2.
\end{equation*}
with probability at least $1 - 3 e^{-t}$. Hence,
\begin{equation*}
\lVert V \hat h_r -  g \rVert_{L^2 (P)}^2
\end{equation*}
is at most
\begin{equation*}
\frac{8 \left (2 C^2 + 8 \lVert k \rVert_\infty^{1/2} C^{3/2} r^{1/2} + \lVert k \rVert_\infty (16 C + 5 \sigma) r \right ) t^{1/2}}{n^{1/2}} + \frac{16 C^2 t}{3 n} + 10 \lVert h_r - g \rVert_\infty^2.
\end{equation*}
Take a sequence of $h_{r, n} \in r B_H$ for $n \geq 1$ with
\begin{equation*}
\lVert h_{r, n} - g \rVert_\infty^2 \downarrow I_\infty (g, r)
\end{equation*}
as $n \to \infty$ and let $A_n \in \mathcal{F}$ be the set with $\prob(A_n) \geq 1 - 3 e^{-t}$ on which the above inequality holds with $h_r = h_{r, n}$. Since the $A_n$ are non-increasing sets, we have that
\begin{equation*}
\prob \left ( \bigcap_{n = 1}^\infty A_n \right ) = \lim_{n \to \infty} \prob(A_n) \geq 1 - 3 e^{-t}.
\end{equation*}
The result follows with
\begin{equation*}
D_1 = 16 C^2, \; D_2 = 64 \lVert k \rVert_\infty^{1/2} C^{3/2}, \; D_3 = 8 \lVert k \rVert_\infty (16 C + 5 \sigma) \text{ and } D_4 = 16 C^2 / 3.
\end{equation*} \hfill \BlackBox

We assume \ref{g3} to prove Theorem \ref{tProbBddInterBound}.

{\bf Proof of Theorem \ref{tProbBddInterBound}}
The initial bound follows from Theorem \ref{tProbBddBound} and \eqref{eLinftyApproxInter} with $D_5 = 10 B^{2/(1-\beta)}$. Based on this bound, setting
\begin{equation*}
r = D_6 t^{- (1-\beta)/(2 (1+\beta))} n^{(1-\beta)/(2 (1+\beta))}
\end{equation*}
gives
\begin{equation*}
\lVert V \hat h_r -  g \rVert_{L^2 (P)}^2
\end{equation*}
is at most
\begin{align*}
&\left (D_3 D_6 + D_5 D_6^{- 2\beta/(1-\beta)} \right ) t^{\beta/(1+\beta)} n^{- \beta/(1+\beta)} + D_2 D_6^{1/2} t^{(1+3\beta)/(4(1+\beta))} n^{- (1+3\beta)/(4(1+\beta))} \\
+ \; &D_1 t^{1/2} n^{-1/2} + D_4 t n^{-1}.
\end{align*}
Hence, the next bound follows with
\begin{equation*}
D_7 = D_3 D_6 + D_5 D_6^{- 2\beta/(1-\beta)} \text{ and } D_8 = D_2 D_6^{1/2}.
\end{equation*} \hfill \BlackBox

\section{Proof of High-Probability Bound for Validation}

We need to introduce some new notation for the next result. Let $U$ and $V$ be random variables on $(\Omega, \mathcal{F})$. Then
\begin{align*}
\lVert U \rVert_{\psi_2} &= \inf \{ a \in (0, \infty) : \expect \psi_2 (\lvert U \rvert/a) \leq 1 \}, \\
\lVert U \vert V \rVert_{\psi_2} &= \inf \{ a \in (0, \infty) : \expect(\psi_2 (\lvert U \rvert/a) \vert V) \leq 1 \text{ almost surely} \},
\end{align*}
where $\psi_2 (x) = \exp(x^2) - 1$ for $x \in \real$. Note that these infima are attained by the monotone convergence theorem. Exercise 5 of Section 2.3 of \citet{gine2015mathematical} shows that $\lVert U \rVert_{\psi_2}$ is a norm on the space of $U$ such that $\lVert U \rVert_{\psi_2} < \infty$ and $\lVert U \vert V \rVert_{\psi_2}$ is a norm on the space of $U$ such that $\lVert U \vert V \rVert_{\psi_2} < \infty$.

We bound the distance between $V \hat h_{\hat r}$ and $V \hat h_{r_0}$ in the $L^2(\tilde P_{\tilde n})$ norm for $r_0 \in R$.

\begin{lemma} \label{lProbValEmpNorm}
Assume \ref{H} and \ref{tY}. Let $r_0 \in R$ and $t \geq 1$. With probability at least $1 - 2 e^{-t}$, we have
\begin{equation*}
\lVert V \hat h_{\hat r} -  V \hat h_{r_0} \rVert_{L^2 (\tilde P_{\tilde n})}^2
\end{equation*}
is at most
\begin{equation*}
\frac{292 C \tilde \sigma t^{1/2}}{\tilde n^{1/2}} \left ( \left ( 2 \log \left (1 + \frac{\lVert k \rVert_\infty^2 \rho^2}{8 C^2} \right ) \right )^{1/2} + \pi^{1/2} \right ) + \frac{24 C^2 t^{1/2}}{\tilde n^{1/2}}  + 4 \lVert V \hat h_{r_0} - g \rVert_{L^2 (P)}^2.
\end{equation*}
\end{lemma}

\begin{proof}
By Lemma \ref{lIneq} with $A = F$ and $n$, $X$, $Y$ and $P_n$ replaced by $\tilde n$, $\tilde X$, $\tilde Y$ and $\tilde P_{\tilde n}$, we have
\begin{equation*}
\lVert V \hat h_{\hat r} -  V \hat h_{r_0} \rVert_{L^2 (\tilde P_{\tilde n})}^2 \leq \frac{4}{\tilde n} \sum_{i = 1}^{\tilde n} (\tilde Y_i - g(\tilde X_i)) (V \hat h_{\hat r} (\tilde X_i) -  V \hat h_{r_0} (\tilde X_i)) + 4 \lVert V \hat h_{r_0} - g \rVert_{L^2 (\tilde P_{\tilde n})}^2.
\end{equation*}
We now bound the right-hand side. We have
\begin{equation*}
\lVert V \hat h_{r_0} - g \rVert_{L^2 (\tilde P_{\tilde n})}^2 = \frac{1}{\tilde n} \sum_{i = 1}^{\tilde n} \left ((V \hat h_{r_0} (\tilde X_i) - g(\tilde X_i))^2 - \lVert V \hat h_{r_0} - g \rVert_{L^2 (P)}^2 \right ) + \lVert V \hat h_{r_0} - g \rVert_{L^2 (P)}^2.
\end{equation*}
Since
\begin{equation*}
\left \lvert (V \hat h_{r_0} (\tilde X_i) - g(\tilde X_i))^2 - \lVert V \hat h_{r_0} - g \rVert_{L^2 (P)}^2 \right \rvert \leq 4 C^2
\end{equation*}
for all $1 \leq i \leq \tilde n$, we find
\begin{equation*}
\lVert V \hat h_{r_0} - g \rVert_{L^2 (\tilde P_{\tilde n})}^2 - \lVert V \hat h_{r_0} - g \rVert_{L^2 (P)}^2 > t
\end{equation*}
with probability at most
\begin{equation*}
\exp \left ( - \frac{\tilde n t^2}{32 C^4} \right ).
\end{equation*}
by Hoeffding's inequality. Therefore, we have
\begin{equation*}
\lVert V \hat h_{r_0} - g \rVert_{L^2 (\tilde P_{\tilde n})}^2 - \lVert V \hat h_{r_0} - g \rVert_{L^2 (P)}^2 \leq \frac{32^{1/2} C^2 t^{1/2}}{\tilde n^{1/2}} \leq \frac{6 C^2 t^{1/2}}{\tilde n^{1/2}} 
\end{equation*}
with probability at least $1 - e^{-t}$. Now let $f_0 = V \hat h_{r_0}$ and
\begin{equation*}
W(f) = \frac{1}{\tilde n} \sum_{i = 1}^{\tilde n}  (\tilde Y_i - g(\tilde X_i)) (f(\tilde X_i) -  f_0 (\tilde X_i))
\end{equation*}
for $f \in F$. $W$ is $\tilde \sigma^2 \lVert \cdot \rVert_\infty^2/\tilde n$-subgaussian given $\tilde X$ and separable on $(F, \tilde \sigma \lVert \cdot \rVert_\infty/\tilde n^{1/2})$ by Lemma \ref{lGaussProc}. The diameter of $(F, \tilde \sigma \lVert \cdot \rVert_\infty/\tilde n^{1/2})$ is
\begin{equation*}
D = \sup_{f_1, f_2 \in F} \tilde \sigma \lVert f_1 - f_2 \rVert_\infty/\tilde n^{1/2} \leq 2 C \tilde \sigma/\tilde n^{1/2}.
\end{equation*}
From Lemma \ref{lCoverInt}, we have
\begin{align*}
\int_0^\infty ( \log(N(F, \tilde \sigma \lVert \cdot \rVert_\infty/\tilde n^{1/2}, \eps)))^{1/2} d\eps &= \int_0^\infty ( \log(N(F, \lVert \cdot \rVert_\infty, \tilde n^{1/2} \eps / \tilde \sigma)))^{1/2} d\eps \\
&= \frac{\tilde \sigma}{\tilde n^{1/2}} \int_0^\infty ( \log(N(F, \lVert \cdot \rVert_\infty, u)))^{1/2} du
\end{align*}
is finite. Hence, by Exercise 1 of Section 2.3 of \citet{gine2015mathematical} and Lemma \ref{lCoverInt}, we have
\begin{equation*}
\left \lVert \sup_{f \in F} \lvert W(f) \rvert \middle \vert \tilde X, X, Y \right \rVert_{\psi_2}
\end{equation*}
is at most
\begin{align*}
&\left \lVert W(f_0) \middle \vert \tilde X, X, Y \right \rVert_{\psi_2} + 1536^{1/2} \int_0^{2 C \tilde \sigma/\tilde n^{1/2}} ( \log N(F, \tilde \sigma \lVert \cdot \rVert_\infty/\tilde n^{1/2}, \eps))^{1/2} d\eps \\
= \; &1536^{1/2} \int_0^{2 C \tilde \sigma/\tilde n^{1/2}} ( \log N(F, \lVert \cdot \rVert_\infty, \tilde n^{1/2} \eps / \tilde \sigma))^{1/2} d\eps \\
= \; &\frac{1536^{1/2} \tilde \sigma}{\tilde n^{1/2}} \int_0^{2 C} ( \log N(F, \lVert \cdot \rVert_\infty, u))^{1/2} du \\
\leq \; &\frac{1536^{1/2} \tilde \sigma}{\tilde n^{1/2}} \left ( 2 \left ( \log \left (1 + \frac{\lVert k \rVert_\infty^2 \rho^2}{8 C^2} \right ) \right )^{1/2} C + (2 \pi)^{1/2} C \right ) \\
= \; &\frac{3072^{1/2} C \tilde \sigma}{\tilde n^{1/2}} \left ( \left ( 2 \log \left (1 + \frac{\lVert k \rVert_\infty^2 \rho^2}{8 C^2} \right ) \right )^{1/2} + \pi^{1/2} \right ),
\end{align*}
noting $W(f_0) = 0$. By Chernoff bounding, we have $\sup_{f \in F} \lvert W(f) \rvert$ is at most
\begin{align*}
&\frac{3072^{1/2} C \tilde \sigma (t + \log(2))^{1/2}}{\tilde n^{1/2}} \left ( \left ( 2 \log \left (1 + \frac{\lVert k \rVert_\infty^2 \rho^2}{8 C^2} \right ) \right )^{1/2} + \pi^{1/2} \right ) \\
\leq \; & \frac{73 C \tilde \sigma t^{1/2}}{\tilde n^{1/2}} \left ( \left ( 2 \log \left (1 + \frac{\lVert k \rVert_\infty^2 \rho^2}{8 C^2} \right ) \right )^{1/2} + \pi^{1/2} \right )
\end{align*}
with probability at least $1 - e^{-t}$. In particular,
\begin{equation*}
\frac{1}{\tilde n} \sum_{i = 1}^{\tilde n} (\tilde Y_i - g(\tilde X_i)) (V \hat h_{\hat r} (\tilde X_i) -  V \hat h_{r_0} (\tilde X_i))
\end{equation*}
is at most
\begin{equation*}
\frac{73 C \tilde \sigma t^{1/2}}{\tilde n^{1/2}} \left ( \left ( 2 \log \left (1 + \frac{\lVert k \rVert_\infty^2 \rho^2}{8 C^2} \right ) \right )^{1/2} + \pi^{1/2} \right )
\end{equation*}
with probability at least $1 - e^{-t}$. The result follows.
\end{proof}
The following lemma is useful for moving the bound on the distance between $V \hat h_{\hat r}$ and $V \hat h_{r_0}$ from the $L^2(\tilde P_{\tilde n})$ norm to the $L^2 (P)$ norm for $r_0 \in R$.

\begin{lemma} \label{lProbValSwapNorm}
Assume \ref{H}. Let $r_0 \in R$, $f_0 = V \hat h_{r_0}$ and $t \geq 1$. With probability at least $1 - e^{-t}$, we have
\begin{equation*}
\sup_{f \in F} \left \lvert \lVert f -  f_0 \rVert_{L^2 (\tilde P_{\tilde n})}^2 - \lVert f -  f_0 \rVert_{L^2 (P)}^2 \right \rvert
\end{equation*}
is at most
\begin{equation*}
\frac{10 C^2 t^{1/2}}{\tilde n^{1/2}} \left ( 1 + 32 \left ( \left ( 2 \log \left (2 + \frac{\lVert k \rVert_\infty^2 \rho^2}{C^2} \right ) \right )^{1/2} + \pi^{1/2} \right ) \right ) + \frac{8 C^2 t}{3 \tilde n}.
\end{equation*}
\end{lemma}

\begin{proof}
Let $A = \{ f - f_0 : f \in F \}$ and
\begin{equation*}
Z = \sup_{f \in F} \left \lvert \lVert f -  f_0 \rVert_{L^2 (\tilde P_{\tilde n})}^2 - \lVert f -  f_0 \rVert_{L^2 (P)}^2 \right \rvert.
\end{equation*}
Then $A \subs L^\infty$ is separable by Lemma \ref{lContEst}. Furthermore, $\lVert f - f_0 \rVert_\infty \leq 2 C$ for all $f \in F$. By Lemma \ref{lTal} with $n$ and $P_n$ replaced by $\tilde n$ and $\tilde P_{\tilde n}$, we have
\begin{equation*}
Z \leq \expect(Z) + \left ( \frac{32 C^4 t}{\tilde n} + \frac{16 C^2 \expect(Z) t}{\tilde n} \right )^{1/2} + \frac{8 C^2 t}{3 \tilde n}
\end{equation*}
with probability at least $1 - e^{-t}$. By Lemma \ref{lValSwapNorm}, we have
\begin{equation*}
\expect(Z) \leq \frac{64 C^2}{\tilde n^{1/2}} \left ( \left ( 2 \log \left (2 + \frac{\lVert k \rVert_\infty^2 \rho^2}{C^2} \right ) \right )^{1/2} + \pi^{1/2} \right ).
\end{equation*}
The result follows.
\end{proof}
We move the bound on the distance between $V \hat h_{\hat r}$ and $V \hat h_{r_0}$ from the $L^2(\tilde P_{\tilde n})$ norm to the $L^2 (P)$ norm for $r_0 \in R$.

\begin{corollary} \label{cProbValTrueNorm}
Assume \ref{H} and \ref{tY}. Let $r_0 \in R$ and $t \geq 1$. With probability at least $1 - 3 e^{-t}$, we have
\begin{equation*}
\lVert V \hat h_{\hat r} -  V \hat h_{r_0} \rVert_{L^2 (P)}^2
\end{equation*}
is at most
\begin{align*}
&\frac{10 C (C + \tilde \sigma) t^{1/2}}{\tilde n^{1/2}} \left ( 1 + 32 \left ( \left ( 2 \log \left (2 + \frac{\lVert k \rVert_\infty^2 \rho^2}{C^2} \right ) \right )^{1/2} + \pi^{1/2} \right ) \right ) \\
+ \; &\frac{24 C^2 t^{1/2}}{\tilde n^{1/2}} + \frac{8 C^2 t}{3 \tilde n} + 4 \lVert V \hat h_{r_0} - g \rVert_{L^2 (P)}^2.
\end{align*}
\end{corollary}

\begin{proof}
By Lemma \ref{lProbValEmpNorm}, we have
\begin{equation*}
\lVert V \hat h_{\hat r} -  V \hat h_{r_0} \rVert_{L^2 (\tilde P_{\tilde n})}^2
\end{equation*}
is at most
\begin{equation*}
\frac{292 C \tilde \sigma t^{1/2}}{\tilde n^{1/2}} \left ( \left ( 2 \log \left (1 + \frac{\lVert k \rVert_\infty^2 \rho^2}{8 C^2} \right ) \right )^{1/2} + \pi^{1/2} \right ) + \frac{24 C^2 t^{1/2}}{\tilde n^{1/2}}  + 4 \lVert V \hat h_{r_0} - g \rVert_{L^2 (P)}^2
\end{equation*}
with probability at least $1 - 2 e^{-t}$. Let $f_0 = V \hat h_{r_0}$. Since $\hat h_{\hat r} \in F$, by Lemma \ref{lProbValSwapNorm} we have
\begin{align*}
&\lVert V \hat h_{\hat r} -  V \hat h_{r_0} \rVert_{L^2 (P)}^2 - \lVert V \hat h_{\hat r} -  V \hat h_{r_0} \rVert_{L^2 (\tilde P_{\tilde n})}^2 \\
\leq \; &\sup_{f \in F} \left \lvert \lVert f -  f_0 \rVert_{L^2 (\tilde P_{\tilde n})}^2 - \lVert f -  f_0 \rVert_{L^2 (P)}^2 \right \rvert \\
\leq \; &\frac{10 C^2 t^{1/2}}{\tilde n^{1/2}} \left ( 1 + 32 \left ( \left ( 2 \log \left (2 + \frac{\lVert k \rVert_\infty^2 \rho^2}{C^2} \right ) \right )^{1/2} + \pi^{1/2} \right ) \right ) + \frac{8 C^2 t}{3 \tilde n}
\end{align*}
with probability at least $1 - e^{-t}$. The result follows.
\end{proof}
We bound the distance between $V \hat h_{\hat r}$ and $g$ in the $L^2(P)$ norm to prove Theorem \ref{tProbValBound}.

{\bf Proof of Theorem \ref{tProbValBound}}
We have
\begin{align*}
\lVert V \hat h_{\hat r} - g \rVert_{L^2 (P)}^2 &\leq \left ( \lVert V \hat h_{\hat r} - V h_{r_0} \rVert_{L^2 (P)}^2 + \lVert V h_{r_0} - g \rVert_{L^2 (P)}^2 \right )^2 \\
&\leq 2 \lVert V \hat h_{\hat r} - V h_{r_0} \rVert_{L^2 (P)}^2 + 2 \lVert V h_{r_0} - g \rVert_{L^2 (P)}^2.
\end{align*}
By Corollary \ref{cProbValTrueNorm}, we have
\begin{equation*}
\lVert V \hat h_{\hat r} -  V \hat h_{r_0} \rVert_{L^2 (P)}^2
\end{equation*}
is at most
\begin{align*}
&\frac{10 C (C + \tilde \sigma) t^{1/2}}{\tilde n^{1/2}} \left ( 1 + 32 \left ( \left ( 2 \log \left (2 + \frac{\lVert k \rVert_\infty^2 \rho^2}{C^2} \right ) \right )^{1/2} + \pi^{1/2} \right ) \right ) \\
+ \; &\frac{24 C^2 t^{1/2}}{\tilde n^{1/2}} + \frac{8 C^2 t}{3 \tilde n} + 4 \lVert V \hat h_{r_0} - g \rVert_{L^2 (P)}^2.
\end{align*}
with probability at least $1 - 3 e^{-t}$. Hence,
\begin{equation*}
\lVert V \hat h_{\hat r} -  g \rVert_{L^2 (P)}^2
\end{equation*}
is at most
\begin{align*}
&\frac{20 C (C + \tilde \sigma) t^{1/2}}{\tilde n^{1/2}} \left ( 1 + 32 \left ( \left ( 2 \log \left (2 + \frac{\lVert k \rVert_\infty^2 \rho^2}{C^2} \right ) \right )^{1/2} + \pi^{1/2} \right ) \right ) \\
+ \; &\frac{48 C^2 t^{1/2}}{\tilde n^{1/2}} + \frac{16 C^2 t}{3 \tilde n} + 10 \lVert V \hat h_{r_0} - g \rVert_{L^2 (P)}^2.
\end{align*}
The result follows. \hfill \BlackBox

We assume the conditions of Theorem \ref{tProbBddInterBound} to prove Theorem \ref{tProbValInterBound}.

{\bf Proof of Theorem \ref{tProbValInterBound}}
If we assume \ref{R1}, then $r_0 = a n^{(1-\beta)/(2 (1+\beta))} \in R$ and
\begin{equation*}
\lVert V \hat h_{r_0} -  g \rVert_{L^2 (P)}^2
\end{equation*}
is at most
\begin{align*}
&\left ( D_3 + D_4 a^{1/2} n^{(1-\beta)/(4 (1+\beta))} + D_5 a n^{(1-\beta)/(2 (1+\beta))} \right ) t^{1/2} n^{- 1/2} \\
+ \; &D_6 t n^{-1} + D_7 a^{- 2 \beta/(1-\beta)} n^{- \beta/(1+\beta)}
\end{align*}
with probability at least $1 - 3 e^{-t}$ for some constants $D_3, D_4, D_5, D_6, D_7 > 0$ not depending on $n$, $\tilde n$ or $t$ by Theorem \ref{tProbBddInterBound}. If we assume \ref{R2}, then there is at least one $r_0 \in R$ such that
\begin{equation*}
a n^{(1-\beta)/(2 (1+\beta))} \leq r_0 < a n^{(1-\beta)/(2 (1+\beta))} + b
\end{equation*}
and
\begin{equation*}
\lVert V \hat h_{r_0} -  g \rVert_{L^2 (P)}^2
\end{equation*}
is at most
\begin{align*}
&\left ( D_3 + D_4 r_0^{1/2} + D_5 r_0 \right ) t^{1/2} n^{- 1/2} + D_6 t n^{-1} + D_7 r_0^{- 2 \beta/(1-\beta)} \\
\leq \; &\left ( D_3 + D_4 \left (a^{1/2} n^{(1-\beta)/(4 (1+\beta))} + b^{1/2} \right ) + D_5 \left (a n^{(1-\beta)/(2 (1+\beta))} + b \right ) \right ) t^{1/2} n^{- 1/2} \\
+ \; &D_6 t n^{-1} + D_7a^{- 2 \beta/(1-\beta)} n^{- \beta/(1+\beta)}
\end{align*}
with probability at least $1 - 3 e^{-t}$ by Theorem \ref{tProbBddInterBound}. In either case,
\begin{equation*}
\lVert V \hat h_{r_0} -  g \rVert_{L^2 (P)}^2 \leq D_8 t^{1/2} n^{-\beta/(1+\beta)} + D_9 t n^{-1}
\end{equation*}
for some constants $D_8, D_9 > 0$ not depending on $n$, $\tilde n$ or $t$. By Theorem \ref{tProbValBound}, we have
\begin{equation*}
\lVert V \hat h_{\hat r} -  g \rVert_{L^2 (P)}^2 \leq D_{10} t^{1/2} \log(n)^{1/2} \tilde n^{-1/2} + D_{11} t \tilde n^{-1} + 10 D_8 t^{1/2} n^{-\beta/(1+\beta)} + 10 D_9 t n^{-1}
\end{equation*}
with probability at least $1 - 6 e^{-t}$ for some constants $D_{10}, D_{11} > 0$ not depending on $n$, $\tilde n$ or $t$. Since $\tilde n$ increases at least linearly in $n$, there exists some constant $D_{12} > 0$ such that $\tilde n \geq D_{12} n$. We then have
\begin{equation*}
\lVert V \hat h_{\hat r} -  g \rVert_{L^2 (P)}^2
\end{equation*}
is at most
\begin{align*}
&D_{12}^{-1/2} D_{10} t^{1/2} \log(n)^{1/2} n^{-1/2} + D_{12}^{-1} D_{11} t n^{-1} + 10 D_8 t^{1/2} n^{-\beta/(1+\beta)} + 10 D_9 t n^{-1} \\
\leq \; &D_1 t^{1/2} n^{-\beta/(1+\beta)} + D_2 t n^{-1}
\end{align*}
for some constants $D_1, D_2 > 0$ not depending on $n$, $\tilde n$ or $t$. \hfill \BlackBox

\section{Estimator Calculation and Measurability}

The following result is essentially Theorem 2.1 from \citet{quintana2014measurable}. The authors show that that a strictly-positive-definite matrix which is a $(\comp^{n \times n}, \bor(\comp^{n \times n}))$-valued measurable matrix on $(\Omega, \mathcal{F})$ can be diagonalised by an unitary matrix and a diagonal matrix which are both $(\comp^{n \times n}, \bor(\comp^{n \times n}))$-valued measurable matrices on $(\Omega, \mathcal{F})$. The result holds for non-negative-definite matrices by adding the identity matrix before diagonalisation and subtracting it afterwards. Furthermore, the construction of the unitary matrix produces a matrix with real entries, which is to say an orthogonal matrix, when the strictly-positive-definite matrix has real entries.

\begin{lemma} \label{lDiag}
Let $M$ be a non-negative-definite matrix which is an $(\real^{n \times n}, \bor(\real^{n \times n}))$-valued measurable matrix on $(\Omega, \mathcal{F})$. There exist an orthogonal matrix $A$ and a diagonal matrix $D$ which are both $(\real^{n \times n}, \bor(\real^{n \times n}))$-valued measurable matrices on $(\Omega, \mathcal{F})$ such that $M = A D A^{\trans}$.
\end{lemma}

We prove Lemma \ref{lCalcEst}.

{\bf Proof of Lemma \ref{lCalcEst}}
Let $H_n = \spn \{ k_{X_i} : 1 \leq i \leq n \}$. The subspace $H_n$ is closed in $H$, so there is an orthogonal projection $Q : H \to H_n$. Since $f - Q f \in H_n^\bot$ for all $f \in H$, we have
\begin{equation*}
f(X_i) - (Q f)(X_i) = \langle f - Q f, k_{X_i} \rangle = 0
\end{equation*}
for all $1 \leq i \leq n$. Hence,
\begin{align*}
\inf_{f \in r B_H} \frac{1}{n} \sum_{i = 1}^n (f(X_i) - Y_i)^2 &= \inf_{f \in r B_H} \frac{1}{n} \sum_{i = 1}^n ((Q f)(X_i) - Y_i)^2 \\
&= \inf_{f \in (r B_H) \cap H_n} \frac{1}{n} \sum_{i = 1}^n (f(X_i) - Y_i)^2.
\end{align*}
Let $f \in (r B_H) \cap H_n$ and write
\begin{equation*}
f = \sum_{i = 1}^n a_i k_{X_i}
\end{equation*}
for some $a \in \real^n$. Then
\begin{equation*}
\frac{1}{n} \sum_{i = 1}^n (f(X_i) - Y_i)^2 = n^{-1} (K a - Y)^{\trans} (K a - Y)
\end{equation*}
and $\lVert f \rVert_H^2 = a^{\trans} K a$, so we can write the norm constraint as $a^{\trans} K a + s = r^2$, where $s \geq 0$ is a slack variable. The Lagrangian can be written as
\begin{align*}
L(a, s; \mu) &= n^{-1} (K a - Y)^{\trans} (K a - Y) + \mu (a^{\trans} K a + s - r^2) \\
&= a^{\trans} ( n^{-1} K^2 + \mu K) a - 2 n^{-1} Y^{\trans} K a + \mu s + n^{-1} Y^{\trans} Y - \mu r^2,
\end{align*}
where $\mu$ is the Lagrangian multiplier for the norm constraint. We seek to minimise the Lagrangian for a fixed value of $\mu$. Note that we require $\mu \geq 0$ for the Lagrangian to have a finite minimum, due to the term in $s$. We have
\begin{equation*}
\frac{\partial L}{\partial a} = 2 ( n^{-1} K^2 + \mu K) a - 2 n^{-1} K Y.
\end{equation*}
This being 0 is equivalent to $K ((K + n \mu I) a - Y) = 0$.

Since the kernel $k$ is a measurable function on $(S \times S, \mathcal{S} \otimes \mathcal{S})$ and the $X_i$ are $(S, \mathcal{S})$-valued random variables on $(\Omega, \mathcal{F})$, we find that $K$ is an $(\real^{n \times n}, \bor(\real^{n \times n}))$-valued measurable matrix on $(\Omega, \mathcal{F})$. Furthermore, since the kernel $k$ takes real values and is non-negative definite, $K$ is non-negative definite with real entries. By Lemma \ref{lDiag}, there exist an orthogonal matrix $A$ and a diagonal matrix $D$ which are both $(\real^{n \times n}, \bor(\real^{n \times n}))$-valued measurable matrices on $(\Omega, \mathcal{F})$ such that $K = A D A^{\trans}$. Note that the diagonal entries of $D$ must be non-negative and we may assume that they are non-increasing. Inserting this diagonalisation into $K ((K + n \mu I) a - Y) = 0$ gives
\begin{equation*}
A D ((D + n \mu I) A^{\trans} a - A^{\trans} Y) = 0.
\end{equation*}
Since $A$ has the inverse $A^{\trans}$, this is equivalent to
\begin{equation*}
D ((D + n \mu I) A^{\trans} a - A^{\trans} Y) = 0.
\end{equation*}
This in turn is equivalent to
\begin{equation*}
(A^{\trans} a)_i = (D_{i, i} + n \mu)^{-1} (A^{\trans} Y)_i
\end{equation*}
for $1 \leq i \leq m$. The same $f$ is produced for all such $a$, because if $w$ is the difference between two such $a$, then $(A^{\trans} w)_i = 0$ for $1 \leq i \leq m$ and the squared $H$ norm of
\begin{equation*}
\sum_{i = 1}^n w_i k_{X_i}
\end{equation*}
is $w^{\trans} K w = w^{\trans} A D A^{\trans} w = 0$. Hence, we are free to set $(A^{\trans} a)_i = 0$ for $m+1 \leq i \leq n$. This uniquely defines $A^{\trans} a$, which in turn uniquely defines $a$, since $A^{\trans}$ has the inverse $A$. Note that this definition of $a$ is measurable on $(\Omega \times [0, \infty), \mathcal{F} \otimes \bor([0, \infty)))$, where $\mu$ varies in $[0, \infty)$.

We now search for a value of $\mu$ such that $a$ and $s$ satisfy the norm constraint. We call this value $\mu(r)$. There are two cases. If
\begin{equation*}
r^2 < \sum_{i = 1}^m D_{i, i}^{-1} (A^{\trans} Y)_i^2,
\end{equation*}
then the $a$ above and $s = 0$ minimise $L$ for $\mu = \mu (r) > 0$ and satisfy the norm constraint, where $\mu(r)$ satisfies
\begin{equation*}
\sum_{i = 1}^m \frac{D_{i, i}}{(D_{i, i} + n \mu(r))^2} (A^{\trans} Y)_i^2 = r^2.
\end{equation*}
Otherwise, the $a$ above and
\begin{equation*}
s =r^2 - \sum_{i = 1}^m D_{i, i}^{-1} (A^{\trans} Y)_i^2 \geq 0
\end{equation*}
minimise $L$ for $\mu = \mu(r) = 0$ and satisfy the norm constraint. Hence, the Lagrangian sufficiency theorem shows
\begin{equation*}
\hat h_r = \sum_{i = 1}^n a_i k_{X_i}
\end{equation*}
for the $a$ above with $\mu = \mu(r)$ for $r > 0$. We also have $\hat h_0 = 0$.

Since $\mu(r) > 0$ is strictly decreasing for
\begin{equation*}
r^2 < \sum_{i = 1}^m D_{i, i}^{-1} (A^{\trans} Y)_i^2
\end{equation*}
and $\mu(r) = 0$ otherwise, we find
\begin{equation*}
\{ \mu(r) \leq \mu \} = \left \{ \sum_{i = 1}^m \frac{D_{i, i}}{(D_{i, i} + n \mu)^2} (A^{\trans} Y)_i^2 \leq r^2 \right \}
\end{equation*}
for $\mu \in [0, \infty)$. Therefore, $\mu(r)$ is measurable on $(\Omega \times [0, \infty), \mathcal{F} \otimes \bor((0, \infty)))$, where $r$ varies in $(0, \infty)$. Hence, the $a$ above with $\mu = \mu(r)$ for $r > 0$ is measurable on $(\Omega \times [0, \infty), \mathcal{F} \otimes \bor((0, \infty)))$, where $r$ varies in $(0, \infty)$. By Lemma 4.25 of \citet{steinwart2008support}, the function $\Phi : S \to H$ by $\Phi(x) = k_x$ is a $(H, \bor(H))$-valued measurable function on $(S, \mathcal{S})$. Hence, $k_{X_i}$ for $1 \leq i \leq n$ are $(H, \bor(H))$-valued random variables on $(\Omega, \mathcal{F})$. Together, these show that $\hat h_r$ is a $(H, \bor(H))$-valued measurable function on $(\Omega \times [0, \infty), \mathcal{F} \otimes \bor([0, \infty)))$, where $r$ varies in $[0, \infty)$, recalling that $\hat h_0$ = 0. \hfill \BlackBox

We prove a continuity result about our estimator.

\begin{lemma} \label{lContEst}
Let $r, s \in [0, \infty)$. We have $\lVert \hat h_r - \hat h_s \rVert_H^2 \leq \lvert r^2 - s^2 \rvert$.
\end{lemma}

\begin{proof}
Recall the diagonalisation of $K = A D A^{\trans}$ from Lemma \ref{lCalcEst}. If $u, v \in \real^n$ and
\begin{equation*}
h_1 = \sum_{i = 1}^n u_i k_{X_i} \text{ and } h_2 = \sum_{i = 1}^n v_i k_{X_i},
\end{equation*}
then $\langle h_1, h_2 \rangle_H = u^{\trans} K v = (A^{\trans} u)^{\trans} D (A^{\trans} v)$. Let $s > r$. If $r > 0$ then, by Lemma \ref{lCalcEst}, we have
\begin{align*}
\langle \hat h_r, \hat h_s \rangle_H &= \sum_{i = 1}^m \frac{D_{i, i}}{(D_{i, i} + n \mu(r)) (D_{i, i} + n \mu(s))} (A^{\trans} Y)_i^2 \\
&\geq \sum_{i = 1}^m \frac{D_{i, i}}{(D_{i, i} + n \mu(r))^2} (A^{\trans} Y)_i^2 \\
&= \lVert \hat h_r \rVert_H^2.
\end{align*}
Furthermore, again by Lemma \ref{lCalcEst}, if $\mu(r) > 0$ then $\lVert \hat h_r \lVert_H^2 = r^2$ and
\begin{align*}
\lVert \hat h_r - \hat h_s \rVert_H^2 &= \lVert \hat h_r \rVert_H^2 + \lVert \hat h_s \rVert_H^2 - 2 \langle \hat h_r, \hat h_s \rangle_H \\
&\leq \lVert \hat h_s \rVert_H^2 - \lVert \hat h_r \rVert_H^2 \\
&= \lVert \hat h_s \rVert_H^2 - r^2 \\
&\leq s^2 - r^2.
\end{align*}
Otherwise, $\mu(r) = 0$ and so $\mu(s) = 0$ by Lemma \ref{lCalcEst}, which means $\hat h_r = \hat h_s$. If $r = 0$ then $\hat h_r = 0$ and $\lVert \hat h_r - \hat h_s \rVert_H^2 = \lVert \hat h_s \rVert_H^2 \leq s^2$. Hence, whenever $r < s$, we have $\lVert \hat h_r - \hat h_s \rVert_H^2 \leq s^2 - r^2$. The result follows.
\end{proof}
We also have the estimator $\hat r$ when performing validation.

\begin{lemma} \label{lValMeasEst}
We have that $\hat r$ is a random variable on $(\Omega, \mathcal{F})$.
\end{lemma}

\begin{proof}
Let
\begin{equation*}
W(s) = \frac{1}{\tilde n} \sum_{i = 1}^{\tilde n} (V \hat h_s (\tilde X_i) - \tilde Y_i)^2
\end{equation*}
for $s \in R$. Note that $W(s)$ is a random variable on $(\Omega, \mathcal{F})$ and continuous in $s$ by Lemma \ref{lContEst}. Since $R \subs \real$, it is separable. Let $R_0$ be a countable dense subset of $R$. Then $\inf_{s \in R} W(s) = \inf_{s \in R_0} W(s)$ is a random variable on $(\Omega, \mathcal{F})$ as the right-hand side is the infimum of countably many random variables on $(\Omega, \mathcal{F})$. Let $r \in [0, \rho]$. By the definition of $\hat r$, we have
\begin{equation*}
\{ \hat r \leq r \} = \bigcup_{s \in R \cap [0, r]} \{ W(s) \leq \inf_{t \in R} W(t) \}.
\end{equation*}
Since $R \cap [0, r] \subs \real$, it is separable. Let $A_r$ be a countable dense subset of $R \cap [0, r]$. By the sequential compactness of $R \cap [0, r]$ and continuity of $W(s)$, we have
\begin{equation*}
\{ \hat r \leq r \} = \bigcap_{a = 1}^\infty \bigcup_{s \in A_r} \{ W(s) \leq \inf_{t \in R} W(t) + a^{-1} \}.
\end{equation*}
This set is an element of $\mathcal{F}$.
\end{proof}

\section{Subgaussian Random Variables}

We need the definition of a sub-$\sigma$-algebra for the next result. The $\sigma$-algebra $\mathcal{G}$ is a sub-$\sigma$-algebra of the $\sigma$-algebra $\mathcal{F}$ if $\mathcal{G} \subs \mathcal{F}$. The following lemma relates a quadratic form of subgaussians to that of centred normal random variables.

\begin{lemma} \label{lSubgauss}
Let $\eps_i$ for $1 \leq i \leq n$ be random variables on $(\Omega, \mathcal{F}, \prob)$ which are independent conditional on some sub-$\sigma$-algebra $\mathcal{G} \subs \mathcal{F}$ and let
\begin{equation*}
\expect( \exp(t \eps_i) \vert \mathcal{G}) \leq \exp( \sigma^2 t^2 / 2)
\end{equation*}
almost surely for all $t \in \real$. Also, let $\delta_i$ for $1 \leq i \leq n$ be random variables on $(\Omega, \mathcal{F}, \prob)$ which are independent of each other and $\mathcal{G}$ with $\delta_i \sim \norm(0, \sigma^2)$. Let $M$ be an $n \times n$ non-negative-definite matrix which is an $(\real^{n \times n}, \bor(\real^{n \times n}))$-valued measurable matrix on $(\Omega, \mathcal{G})$. We have
\begin{equation*}
\expect( \exp(z \eps^{\trans} M \eps ) \vert \mathcal{G}) \leq \expect( \exp(z \delta^{\trans} M \delta ) \vert \mathcal{G})
\end{equation*}
almost surely for all $z \geq 0$.
\end{lemma}

\begin{proof}
This proof method uses techniques from the proof of Lemma 9 of \citet{abbasi2011improved}. We have
\begin{equation*}
\expect( \exp(t_i \eps_i / \sigma) \vert \mathcal{G}) \leq \exp( t_i^2 / 2)
\end{equation*}
almost surely for all $1 \leq i \leq n$ and $t_i \in \real$. Furthermore, the $\eps_i$ are independent conditional on $\mathcal{G}$, so
\begin{equation*}
\expect( \exp(t^{\trans} \eps / \sigma) \vert \mathcal{G}) \leq \exp( \lVert t \rVert_2^2 / 2)
\end{equation*}
almost surely for all $t \in \real^n$. By Lemma \ref{lDiag} with $\mathcal{F}$ replaced by $\mathcal{G}$, there exist an orthogonal matrix $A$ and a diagonal matrix $D$ which are both $(\real^{n \times n}, \bor(\real^{n \times n}))$-valued measurable matrices on $(\Omega, \mathcal{G})$ such that $M = A D A^{\trans}$. Hence, $M$ has a square root $M^{1/2} = A D^{1/2} A^{\trans}$ which is an $(\real^{n \times n}, \bor(\real^{n \times n}))$-valued measurable matrix on $(\Omega, \mathcal{G})$, where $D^{1/2}$ is the diagonal matrix with entries equal to the square root of those of $D$. Note that these entries are non-negative because $M$ is non-negative definite. We can then replace $t$ with $s M^{1/2} u$ for $s \in \real$ and $u \in \real^n$ to get
\begin{equation*}
\expect( \exp(s u^{\trans} M^{1/2} \eps / \sigma) \vert \mathcal{G}) \leq \exp( s^2 \lVert M^{1/2} u \rVert_2^2 / 2)
\end{equation*}
almost surely. Integrating over $u$ with respect to the distribution of $\delta$ gives
\begin{equation*}
\expect( \exp(s \delta^{\trans} M^{1/2} \eps / \sigma) \vert \mathcal{G}) \leq \expect( \exp( s^2 \delta^{\trans} M \delta / 2) \vert \mathcal{G})
\end{equation*}
almost surely. The moment generating function of $\delta$ gives the left-hand side as being
\begin{equation*}
\expect( \exp(s^2 \eps^{\trans} M \eps / 2) \vert \mathcal{G})
\end{equation*}
almost surely. The result follows.
\end{proof}
Having established this relationship, we can now obtain a probability bound on a quadratic form of subgaussians by using Chernoff bounding. The following result is a conditional subgaussian version of the Hanson--Wright inequality.

\begin{lemma} \label{lSubgaussProb}
Let $\eps_i$ for $1 \leq i \leq n$ be random variables on $(\Omega, \mathcal{F}, \prob)$ which are independent conditional on some sub-$\sigma$-algebra $\mathcal{G} \subs \mathcal{F}$ and let
\begin{equation*}
\expect( \exp(t \eps_i) \vert \mathcal{G}) \leq \exp( \sigma^2 t^2 / 2)
\end{equation*}
almost surely for all $t \in \real$. Let $M$ be an $n \times n$ non-negative-definite matrix which is an $(\real^{n \times n}, \bor(\real^{n \times n}))$-valued measurable matrix on $(\Omega, \mathcal{G})$ and $z \geq 0$. We have
\begin{equation*}
\eps^{\trans} M \eps \leq \sigma^2 \tr(M) + 2 \sigma^2 \lVert M \rVert z + 2 \sigma^2 ( \lVert M \rVert^2 z^2 + \tr(M^2) z )^{1/2}
\end{equation*}
with probability at least $1 - e^{-z}$ almost surely conditional on $\mathcal{G}$. Here, $\lVert M \rVert$ is the operator norm of $M$, which is a random variable on $(\Omega, \mathcal{G})$.
\end{lemma}

\begin{proof}
This proof method follows that of Theorem 3.1.9 of \citet{gine2015mathematical}. By Lemma \ref{lDiag} with $\mathcal{F}$ replaced by $\mathcal{G}$, there exist an orthogonal matrix $A$ and a diagonal matrix $D$ which are both $(\real^{n \times n}, \bor(\real^{n \times n}))$-valued measurable matrices on $(\Omega, \mathcal{G})$ such that $M = A D A^{\trans}$. Let $\delta_i$ for $1 \leq i \leq n$ be random variables on $(\Omega, \mathcal{F}, \prob)$ which are independent of each other and $\mathcal{G}$, with $\delta_i \sim \norm(0, \sigma^2)$. By Lemma \ref{lSubgauss} and the fact that $A^{\trans} \delta$ has the same distribution as $\delta$, we have
\begin{equation*}
\expect( \exp(t \eps^{\trans} M \eps ) \vert \mathcal{G}) \leq \expect( \exp(t \delta^{\trans} M \delta ) \vert \mathcal{G}) = \expect( \exp(t \delta^{\trans} D \delta ) \vert \mathcal{G})
\end{equation*}
almost surely for all $t \geq 0$. Furthermore,
\begin{equation*}
\expect( \exp(t \delta_i^2 / \sigma^2 )) = \int_{- \infty}^\infty \frac{1}{(2 \pi)^{1/2}} \exp(t x^2 - x^2 / 2) dx = \frac{1}{(1 - 2 t)^{1/2}}
\end{equation*}
for $0 \leq t < 1/2$ and $1 \leq i \leq n$, so
\begin{equation*}
\expect( \exp(t (\delta_i^2 / \sigma^2 - 1))) = \exp(- ( \log(1 - 2 t) + 2 t) / 2).
\end{equation*}
We have
\begin{equation*}
- 2 ( \log(1 - 2 t) + 2 t) \leq \sum_{i = 2}^\infty (2 t)^i (2 / i) \leq 4 t^2 / (1 - 2 t)
\end{equation*}
for $0 \leq t < 1/2$. Therefore, since the $\delta_i$ are independent of $\mathcal{G}$, we have
\begin{equation*}
\expect( \exp(t D_{i, i} (\delta_i^2 - \sigma^2)) \vert \mathcal{G}) \leq \exp \left ( \frac{\sigma^4 D_{i, i}^2 t^2}{1 - 2 \sigma^2 D_{i, i} t} \right )
\end{equation*}
almost surely for $0 \leq t < 1/(2 \sigma^2 D_{i, i})$ and $1 \leq i \leq n$. Since the $D_{i, i}$ are random variables on $(\Omega, \mathcal{G})$ and the $D_{i, i} \delta_i$ for $1 \leq i \leq n$ are independent conditional on $\mathcal{G}$, we have
\begin{equation*}
\expect( \exp(t (\delta^{\trans} D \delta - \sigma^2 \tr(D))) \vert \mathcal{G}) \leq \exp \left ( \frac{\sigma^4 \tr(D^2) t^2}{1 - 2 \sigma^2 (\max_i D_{i, i}) t} \right )
\end{equation*}
almost surely for $0 \leq t < 1/(2 \sigma^2 (\max_i D_{i, i}))$. Combining this with $\expect( \exp(t \eps^{\trans} M \eps ) \vert \mathcal{G}) \leq \expect( \exp(t \delta^{\trans} D \delta ) \vert \mathcal{G})$, we find
\begin{equation*}
\expect( \exp(t (\eps^{\trans} M \eps - \sigma^2 \tr(M))) \vert \mathcal{G}) \leq \exp \left ( \frac{\sigma^4 \tr(M^2) t^2}{1 - 2 \sigma^2 \lVert M \rVert t} \right )
\end{equation*}
almost surely for $0 \leq t < 1/(2 \sigma^2 \lVert M \rVert)$. By Chernoff bounding, we have
\begin{equation*}
\eps^{\trans} M \eps - \sigma^2 \tr(M) > s
\end{equation*}
for $s \geq 0$ with probability at most
\begin{equation*}
\exp \left ( \frac{\sigma^4 \tr(M^2) t^2}{1 - 2 \sigma^2 \lVert M \rVert t}  - t s \right )
\end{equation*}
almost surely conditional on $\mathcal{G}$ for $0 \leq t < 1/(2 \sigma^2 \lVert M \rVert)$. Letting
\begin{equation*}
t = \frac{s}{2 \sigma^4 \tr(M^2) + 2 \sigma^2 \lVert M \rVert s}
\end{equation*}
gives the bound
\begin{equation*}
\exp \left ( - \frac{s^2}{4 \sigma^4 \tr(M^2) + 4 \sigma^2 \lVert M \rVert s} \right ).
\end{equation*}
Rearranging gives the result.
\end{proof}

\section{Covering Numbers}

The following lemma gives a bound on the covering numbers of $F$.

\begin{lemma} \label{lCoverNum}
Let $\eps > 0$. We have
\begin{equation*}
N(F, \lVert \cdot \rVert_\infty, \eps) \leq 1 + \frac{\lVert k \rVert_\infty^2 \rho^2}{2 \eps^2}.
\end{equation*}
\end{lemma}

\begin{proof}
Let $a \geq 1$ and $r_i \in R$ and $f_i = V \hat h_{r_i} \in F$ for $1 \leq i \leq a$. Also, let $f = V \hat h_r \in F$ for $r \in R$. Since $V$ is a contraction, we have $\lVert f - f_i \rVert_\infty \leq \eps$ whenever $\lVert \hat h_r - \hat h_{r_i} \rVert_\infty \leq \eps$. By Lemma \ref{lContEst}, we have $\lVert \hat h_r - \hat h_{r_i} \rVert_\infty \leq \eps$ whenever $\lvert r^2 - r_i^2 \rvert \leq \eps^2/\lVert k \rVert_\infty^2$. Hence, if we let $r_i^2 = \eps^2 (2 i - 1) / \lVert k \rVert_\infty^2$ and let $\rho$ be such that
\begin{equation*}
\rho^2 - \eps^2 (2 a - 1) / \lVert k \rVert_\infty^2 \leq \eps^2/\lVert k \rVert_\infty^2,
\end{equation*}
then we find $N(F, \lVert \cdot \rVert_\infty, \eps) \leq a$. Rearranging the above shows that we can choose
\begin{equation*}
a = \left \lceil \frac{\lVert k \rVert_\infty^2 \rho^2}{2 \eps^2} \right \rceil
\end{equation*}
and the result follows.
\end{proof}
We also calculate integrals of these covering numbers.

\begin{lemma} \label{lCoverInt}
Let $a \geq 1$. We have
\begin{equation*}
\int_0^L ( \log(a N(F, \lVert \cdot \rVert_\infty, \eps)))^{1/2} d\eps \leq \left ( \log \left ( \left (1 + \frac{\lVert k \rVert_\infty^2 \rho^2}{2 L^2} \right ) a \right ) \right )^{1/2} L + \left ( \frac{\pi}{2} \right )^{1/2} L
\end{equation*}
for $L \in (0, \infty)$. When $a = 1$, we have
\begin{equation*}
\int_0^L ( \log(N(F, \lVert \cdot \rVert_\infty, \eps)))^{1/2} d\eps \leq 2 \left ( \log \left (1 + \frac{\lVert k \rVert_\infty^2 \rho^2}{8 C^2} \right ) \right )^{1/2} C + (2 \pi)^{1/2} C
\end{equation*}
for $L \in (0, \infty]$.
\end{lemma}

\begin{proof}
Let $L \in (0, \infty)$. Then
\begin{equation*}
\int_0^L ( \log(a N(F, \lVert \cdot \rVert_\infty, \eps)))^{1/2} d\eps \leq \int_0^L \left ( \log \left (a \left (1 + \frac{\lVert k \rVert_\infty^2 \rho^2}{2 \eps^2} \right ) \right ) \right )^{1/2} d\eps
\end{equation*}
by Lemma \ref{lCoverNum}. Changing variables to $u = \eps/L$ gives
\begin{align*}
&L \int_0^1 \left ( \log \left (a \left (1 + \frac{\lVert k \rVert_\infty^2 \rho^2}{2 L^2 u^2} \right ) \right ) \right )^{1/2} du \\
\leq \; &L \int_0^1 \left ( \log \left (a \left (1 + \frac{\lVert k \rVert_\infty^2 \rho^2}{2 L^2} \right ) \frac{1}{u^2} \right ) \right )^{1/2} du \\
= \; &L \int_0^1 \left ( \log \left (a \left (1 + \frac{\lVert k \rVert_\infty^2 \rho^2}{2 L^2} \right ) \right ) + \log \left ( \frac{1}{u^2} \right ) \right )^{1/2} du.
\end{align*}
For $b, c \geq 0$ we have $(b + c)^{1/2} \leq b^{1/2} + c^{1/2}$, so the above is at most
\begin{align*}
&L \int_0^1 \left ( \log \left (a \left (1 + \frac{\lVert k \rVert_\infty^2 \rho^2}{2 L^2} \right ) \right ) \right )^{1/2} du + L \int_0^1 \left ( \log \left ( \frac{1}{u^2} \right ) \right )^{1/2} du \\
= \; &L \left ( \log \left (a \left (1 + \frac{\lVert k \rVert_\infty^2 \rho^2}{2 L^2} \right ) \right ) \right )^{1/2} + L \int_0^1 \left ( \log \left ( \frac{1}{u^2} \right ) \right )^{1/2} du.
\end{align*}
Changing variables to
\begin{equation*}
s = \left ( \log \left ( \frac{1}{u^2} \right ) \right )^{1/2}
\end{equation*}
shows
\begin{align*}
\int_0^1 \left ( \log \left ( \frac{1}{u^2} \right ) \right )^{1/2} du &= \int_0^\infty s^2 \exp(- s^2/2) ds \\
&= \frac{1}{2} \int_{- \infty}^\infty s^2 \exp(- s^2/2) ds \\
&= \left ( \frac{\pi}{2} \right )^{1/2},
\end{align*}
since the last integral is a multiple of the variance of an $\norm(0, 1)$ random variable. The first result follows. Note that $N(F, \lVert \cdot \rVert_\infty, \eps) = 1$ whenever $\eps \geq 2 C$, as the ball of radius $2 C$ about any point in $F$ is the whole of $F$. Hence, when $a = 1$, we have
\begin{align*}
\int_0^L ( \log(N(F, \lVert \cdot \rVert_\infty, \eps)))^{1/2} d \eps &\leq \int_0^\infty ( \log(N(F, \lVert \cdot \rVert_\infty, \eps)))^{1/2} d\eps \\
&= \int_0^{2 C} ( \log(N(F, \lVert \cdot \rVert_\infty, \eps)))^{1/2} d\eps \\
&\leq 2 \left ( \log \left (1 + \frac{\lVert k \rVert_\infty^2 \rho^2}{8 C^2} \right ) \right )^{1/2} C + (2 \pi)^{1/2} C
\end{align*}
for $L \in (0, \infty]$.
\end{proof}

\bibliography{paper}

\begin{thebibliography}{16}
\providecommand{\natexlab}[1]{#1}
\providecommand{\url}[1]{\texttt{#1}}
\expandafter\ifx\csname urlstyle\endcsname\relax
  \providecommand{\doi}[1]{doi: #1}\else
  \providecommand{\doi}{doi: \begingroup \urlstyle{rm}\Url}\fi

\bibitem[Abbasi-Yadkori et~al.(2011)Abbasi-Yadkori, P\'{a}l, and
  Szepesv\'{a}ri]{abbasi2011improved}
Yasin Abbasi-Yadkori, D\'{a}vid P\'{a}l, and Csaba Szepesv\'{a}ri.
\newblock Improved algorithms for linear stochastic bandits.
\newblock In \emph{Proceedings of the 24th International Conference on Neural
  Information Processing Systems}, pages 2312--2320, 2011.

\bibitem[Bergh and L{\"o}fstr{\"o}m(1976)]{bergh2012interpolation}
J{\"o}ran Bergh and J{\"o}rgen L{\"o}fstr{\"o}m.
\newblock \emph{Interpolation Spaces. {A}n Introduction}.
\newblock Springer--Verlag, Berlin--New York, 1976.

\bibitem[Caponnetto and de~Vito(2007)]{caponnetto2007optimal}
Andrea Caponnetto and Ernesto de~Vito.
\newblock Optimal rates for the regularized least-squares algorithm.
\newblock \emph{Found. Comput. Math.}, 7\penalty0 (3):\penalty0 331--368, 2007.

\bibitem[Fischer and Steinwart(2017)]{fischer2017sobolev}
Simon Fischer and Ingo Steinwart.
\newblock Sobolev norm learning rates for regularized least-squares algorithm.
\newblock \emph{arXiv preprint arXiv:1702.07254}, 2017.

\bibitem[Gin{\'e} and Nickl(2016)]{gine2015mathematical}
Evarist Gin{\'e} and Richard Nickl.
\newblock \emph{Mathematical Foundations of Infinite-Dimensional Statistical
  Models}.
\newblock Cambridge University Press, New York, 2016.

\bibitem[Mendelson and Neeman(2010)]{mendelson2010regularization}
Shahar Mendelson and Joseph Neeman.
\newblock Regularization in kernel learning.
\newblock \emph{Ann. Statist.}, 38\penalty0 (1):\penalty0 526--565, 2010.

\bibitem[Oneto et~al.(2016)Oneto, Ridella, and Anguita]{oneto2016tikhonov}
Luca Oneto, Sandro Ridella, and Davide Anguita.
\newblock Tikhonov, {I}vanov and {M}orozov regularization for support vector
  machine learning.
\newblock \emph{Mach. Learn.}, 103\penalty0 (1):\penalty0 103--136, 2016.

\bibitem[Quintana and Rodr\'{i}guez(2014)]{quintana2014measurable}
Yamilet Quintana and Jos\'{e}~M. Rodr\'{i}guez.
\newblock Measurable diagonalization of positive definite matrices.
\newblock \emph{J. Approx. Theory}, 185:\penalty0 91--97, 2014.

\bibitem[Smale and Zhou(2003)]{smale2003estimating}
Steve Smale and Ding-Xuan Zhou.
\newblock Estimating the approximation error in learning theory.
\newblock \emph{Anal. Appl. (Singap.)}, 1\penalty0 (1):\penalty0 17--41, 2003.

\bibitem[Smale and Zhou(2007)]{smale2007learning}
Steve Smale and Ding-Xuan Zhou.
\newblock Learning theory estimates via integral operators and their
  approximations.
\newblock \emph{Constr. Approx.}, 26\penalty0 (2):\penalty0 153--172, 2007.

\bibitem[Sriperumbudur(2016)]{sriperumbudur2016optimal}
Bharath Sriperumbudur.
\newblock On the optimal estimation of probability measures in weak and strong
  topologies.
\newblock \emph{Bernoulli}, 22\penalty0 (3):\penalty0 1839--1893, 2016.

\bibitem[Steinwart and Christmann(2008)]{steinwart2008support}
Ingo Steinwart and Andreas Christmann.
\newblock \emph{Support Vector Machines}.
\newblock Springer--Verlag, New York, 2008.

\bibitem[Steinwart and Scovel(2012)]{steinwart2012mercer}
Ingo Steinwart and Clint Scovel.
\newblock Mercer's theorem on general domains: On the interaction between
  measures, kernels, and {RKHS}s.
\newblock \emph{Constr. Approx.}, 35\penalty0 (3):\penalty0 363--417, 2012.

\bibitem[Steinwart et~al.(2009)Steinwart, Hush, and
  Scovel]{steinwart2009optimal}
Ingo Steinwart, Don~R. Hush, and Clint Scovel.
\newblock Optimal rates for regularized least squares regression.
\newblock In \emph{The 22nd Conference on Learning Theory}, 2009.

\bibitem[van~der Vaart and Wellner(1996)]{van1996weak}
Aad~W. van~der Vaart and Jon~A. Wellner.
\newblock \emph{Weak Convergence and Empirical Processes}.
\newblock Springer--Verlag, New York, 1996.

\bibitem[Williams(1991)]{williams1991probability}
David Williams.
\newblock \emph{Probability with Martingales}.
\newblock Cambridge University Press, Cambridge, 1991.

\end{thebibliography}

\end{document}